\documentclass{amsart}
\usepackage{xypic}

\makeatletter
\@namedef{subjclassname@2020}{%
  \textup{2020} Mathematics Subject Classification}
\makeatother

\numberwithin{equation}{section}
\usepackage[a4paper,left=3.5cm,right=3.5cm]{geometry}
\usepackage{amsmath}
\usepackage{amsfonts}
\usepackage{amssymb}
\usepackage{amscd}
\usepackage{amsthm}
\usepackage{color}
\usepackage{xcolor}
\usepackage{verbatim}

\usepackage{hyperref}
\theoremstyle{plain}
\newtheorem{theoremint}{Theorem}
\newtheorem{theorem}{Theorem}[section]
\newtheorem{proposition}[theorem]{Proposition}
\newtheorem{lemma}[theorem]{Lemma}
\newtheorem{corollary}[theorem]{Corollary}
\theoremstyle{definition}

\newtheorem{definition}[theorem]{Definition}
\newtheorem{remark}[theorem]{Remark}
\newtheorem{construction}[theorem]{Construction}

\DeclareMathOperator{\Ext}{Ext}
\DeclareMathOperator{\MI}{MI}
\DeclareMathOperator{\Hom}{Hom}

\DeclareMathOperator{\Pic}{Pic}

\newcommand{\ZZ}{\mathbb{Z}}
\newcommand{\PP}{\mathbb{P}}
\newcommand{\Oo}{\mathcal{O}}

\newcommand{\rk}{\mathrm{rk}}

\newcommand{\pp}{{\mathbb P}}
\newcommand{\p}{{{\mathbb P}^1}}
\newcommand{\kk}{{\mathbb K}}
\newcommand\sA{{\mathcal A}}
\newcommand\sB{{\mathcal B}}
\newcommand\sC{{\mathcal C}}
\newcommand\sD{{\mathcal D}}
\newcommand\sE{{\mathcal E}}
\newcommand\sF{{\mathcal F}}
\newcommand\sG{{\mathcal G}}

\newcommand\sI{{\mathcal I}}

\newcommand\sL{{\mathcal L}}

\newcommand\sN{{\mathcal N}}
\newcommand\sO{{\mathcal O}}

\newcommand\sT{{\mathcal T}}
\newcommand\sU{{\mathcal U}}

\def\pee#1{\hbox{$ {\mathbb P}^{#1}$}}

  \def \tab#1{\kern #1 truein}

   \newenvironment{sistema}%
  {\left\lbrace\begin{array}{@{}l@{}}}%
  {\end{array}\right.}

\begin{document}

  \title[H-instanton bundles on three-dimensional polarized projective varieties]{H-instanton bundles on three-dimensional polarized projective varieties
}

\author{V. Antonelli, F. Malaspina }

\keywords{Instanton bundles, Ulrich bundles, rational normal scrolls, moduli spaces, Beilinson spectral sequence}

\subjclass[2020]{Primary: {14J60}; Secondary: {13C14, 14F06, 14F08}}

\thanks{The authors are members of GNSAGA group of INdAM and are supported by the framework of the MIUR grant Dipartimenti di Eccellenza 2018-2022 (E11G18000350001).}

\begin{abstract}
We propose a notion of instanton bundle (called $H$-instanton bundle) on any projective variety of dimension three polarized by a very ample divisor $H$, that naturally generalizes the ones on $\mathbb{P}^3$ and on the flag threefold $F(0,1,2)$. We briefly discuss the cases of Veronese and Fano threefolds. Then we deal with $H$-instanton bundles $\mathcal{E}$ on three-dimensional rational normal scrolls $S(a_0,a_1,a_2)$. We give a monadic description of $H$-instanton bundles and we prove the existence of $\mu$-stable $H$-instanton bundles on $S(a_0,a_1,a_2)$ for any admissible charge $k=c_2(\mathcal{E})H$. Then we deal in more detail with $S(a,a,b)$ and $S(a_0,a_1,a_2)$ with $a_0+a_1>a_2$ and even degree. Finally we describe a nice component of the moduli space of $\mu$-stable bundles whose points represent $H$-instantons. \end{abstract}
\maketitle
\section*{Introduction}
\emph{Mathematical instanton bundles} on $\pee3$ arose as the algebraic counterpart to the (anti)self-dual connections which are solutions to the Yang-Mills equations on $S^4$.
The link between these two apparently distant objects was achieved through twistor theory, as it was developed by R. Penrose. Indeed $\pee3$ is the twistor space of $S^4$.
Prompted by this fact, mathematical instanton bundles on $\pee3$ have been the object of study for the last three decades, becoming one of the milestones in algebraic geometry. An instanton bundle on $\pee3$ is a stable rank two bundle $\sE$ with $c_1(\sE)=0$ and satisfying the so called \emph{instantonic condition} $$h^1(\pee3,\sE(-2))=0.$$

In \cite{Fa} (see also \cite{Kuz} in the case of index two and \cite{Sa} for the case of the del Pezzo threefold of degree $5$) is given the following definition of instanton bundle on Fano threefold $X$ with Picard number one: let  us write the canonical divisor of $(X,H_X)$ as $\omega_X=-(2q_X+r_X)H_X$ where $q_X\geq 0$ and $0\leq r_X\leq 1$, then an instanton bundle $\sE$ on $X$  is a $\mu$-stable rank two vector bundle with $c_1(\sE)=-r_X$ (hence normalized) and
$$H^1(X,\sE(-q_X))=0.$$ So $\sG=\sE(-q_X)$ satisfies $\sG\cong \sG^\vee\otimes\omega_X$ and by Serre duality $H^1(X,\sG)=H^2(X,\sG)=0$.

Apart from this definition there have been several attempts to generalize instanton bundles to Fano varieties of higher Picard number. In \cite{MMP} the authors gave the following definition of instanton bundle over the Flag variety $F(0,1,2)$:
a rank  two vector bundle $\sE$ on the Fano threefold $F(0,1,2)$ is an instanton bundle of charge $k$ if the following properties hold:
\begin{itemize}
\item $c_1(\sE)=0, \ c_2(\sE)=kh_1h_2$;
\item $h^0(\sE)=0$ and $\sE$ is $\mu$-semistable;
\item $h^1(X,\sE(-H_X))=0$;
\end{itemize}
This definition has been adapted in \cite{AM}, \cite{CCGM} and \cite{CG} to other examples of Fano threefolds with Picard number higher than one.
One can observe that in all the cases, by increasing the second Chern class, the intermediate cohomology modules have more generators and become more complicated. Recall that a coherent sheaf $\sE$  is
aCM (aritmetically Cohen-Macauly) if all the intermediate cohomology groups vanish or equivalently if the module $E$ of global sections of
$\sE$ is a maximal Cohen-Macaulay module. So as the charge of instanton bundles grows, instanton bundles become farther from being aCM.
Among aCM sheaves, a special role is played by Ulrich sheaves. These shaves are characterized by the linearity of the minimal graded free resolution over the polynomial ring of their module of global sections.
Ulrich sheaves, originally studied for computing Chow forms,
conjecturally exist over any variety
(we refer to \cite{ESW}).

In order to introduce a notion of instanton bundles related to the previous ones on any polarized projective threefold $(X, H_X)$ we do not give conditions depending on the index (as in the Fano case), but depending only on the polarization $H_X$. Let $K_X$ be the canonical divisor. we give the following definition:

 A rank  two vector bundle $\sE$ on $X$ is an \emph{$H$-instanton bundle} of charge $k$ if the following properties hold:
\begin{itemize}
\item $c_1(\sE)=2H_X+K_X$ and $\deg(c_2(\sE))=k$;
\item $h^0(X,\sE)=0$ and $\sE$ is $\mu$-semistable;
\item $h^1(X,\sE(-H_X))=0$;
\end{itemize}

In Section 1 we will relate this Definition with the usual definitions on $\pee3$ and on other Fano threefolds.

Notice that when $\rk \Pic (X)=1$, except for isolated case $\sE\cong \sO_{\pee3}^{\oplus 2}$ on $\pee3$, the conditions $h^0(X,\sE)=0$ and $\sE$ is $\mu$-semistable  imply the $\mu$-stability. However when the Picard number is higher than one this is no longer true.  For instance on $F(0,1,2)$ it is natural to consider strictly $\mu$-semistable bundles (see \cite{Don} and \cite[Remark 2.2]{MMP}).  It is also possible to find bundles with $H^0(X,\sE)\not= 0$ but satisfying all the other instantonic conditions (see \cite[Theorem 4.2]{MMP} for $F(0,1,2)$ or \cite[Theorem 2.9]{AM} for $\pee1\times\pee1\times\pee1$). On \cite[Remark 2.10]{AM} it is also given an example of a vector bundle satisfying all the conditions of Definition \ref{hinstanton} but the $\mu$-semistability.

Furthermore it is of great interest the connection between $H$-instantons and Ulrich bundles (even in higher rank cases, see \cite{CoMiRo2020}). Indeed Definition \ref{hinstanton} is related to the notion of Ulrich bundle i.e. a vector bundle $\sE$ such that $h^i(X,\sE_X(tH_X))=0$ for any $i$ and $-1\leq t \leq -\dim(X)$. In fact any $H$-instanton bundles satisfies $h^i(X,\sE(-H_X))=0$ for any $i$, $h^0(X,\sE)=h^3(X,\sE(-2H))=0$ and $\chi(\sE)=\chi(\sE(-2H_X))$. Moreover it is possible to prove that for an $H$-instanton bundle $\sE$ on a smooth polarized threefold we have $h^2(X,\sE)=0$. Thanks to Riemann-Roch this gives us the notion of minimal charge, namely the minimal value of the degree of $c_2(\sE)$ such that $\chi(\sE)=-h^1(X,\sE)\leq 0$. In this case $\sE(H_X)$ turns out to be an Ulrich bundle. When the Picard number is one all the Ulrich bundles are $H$-instanton bundles twisted by $H_X$. Of course if $(X,\Oo_X(H_X))$ does not support any rank two Ulrich bundle (for instance a general hypersurface of degree at least six in $\mathbb P^4$ see \cite{KRR}), we may conclude that there cannot exists an $H$-instanton of minimal charge. In general, regardless of the Picard number, if $X$ supports special rank two Ulrich bundles $\sF$ (i.e. $\det(\sF)=4H_X+K_X$) then such bundles are $H$-instanton bundles twisted by $H_X$.


On the two Fano threefolds with index two and Picard number two all the instanton bundles with $c_2(\sE)$ of minimal degree are a twist of an Ulrich bundle (see \cite{MMP} and \cite{CCGM}). If $\sE$ is an instanton bundle on a Fano threefold $X$ with index two and Picard number one we have $h^1(X,\sE)=c_2(\sE)-2$ and $h^0(X,\sE)=h^2(X,\sE)=h^3(X,\sE)=0$ (see \cite[Section 3]{Fa}). So when the charge is minimal, namely $c_2(\sE)=2$, by Serre duality we have $h^i(X,\sE)=h^i(X,\sE(-1))=h^i(X,\sE(-2))=0$ for any $i$ so $\sE$ is Ulrich up to a twist. On the three-dimensional quadric (the only case of index three) the instanton bundle of minimal charge is the spinor bundle which is Ulrich (see also \cite{CoMi} for instanton bundles on smooth hyperquadrics).

In this paper, thanks to this new definition, we deal with the existence of $H$-instanton bundles on some (infinite) families of threefolds.
In Section 1 we discuss about the relation between classical instanton bundles and $H$-instanton bundles in particular for the cases of Veronese and Fano threefolds. Then, from Section 2 we focus our attention to three-dimensional rational normal scrolls $S$ since this a family of threefolds with infinite different canonical divisors $K_X$ and with a nice description of the derived category of coherent sheaves.  As first results, in Section 3 we give three different monadic descriptions of each $H$-instanton bundle (which can be considered as the analog of the monads on $\pee3$) and we characterize the $H$-instantons $\sE$ such that $c_2(\sE)=kHF$ (i.e. the second Chern class is concentrated in a single degree) showing that they are always realized as an extension of suitable line bundles.

After this, in Section 4 we deal with the existence of $H$-instanton bundles and the description of their moduli spaces. The locus of such stable bundles with a given $c_2$ is an open subspace of the moduli space of stable rank two vector bundles bundles with those invariants. For  a large family of Fano threefolds with Picard number one, in \cite[Theorem A]{Fa} the author proves that the moduli spaces of instanton bundles has a generically smooth irreducible component. An analogous result has been obtained for other Fano threefolds with higher Picard number \cite{AM,CCGM,CG,MMP}. In the case of
$\mathbb P^3$, it has been proved that the moduli space of instanton bundles is affine \cite{CO}, irreducible \cite{T1,T2} and smooth \cite{JV}. The rationality is still an open problem in general, being settled only for  low charges \cite{Ba2,Har2,ES,Kat}.

In the case of three-dimensional rational normal scrolls $S$, starting from the minimal charge case which is represented by Ulrich bundles on $S$, we are able to prove the following theorem.
\begin{theoremint}\label{teo1intro}
Let $S=S(a_0,a_1,a_2)$ be a rational normal scroll of degree $c$. Then for every $\alpha\geq 1-c$ there exists a $\mu$-stable $H$-instanton bundle $\sE$ with $c_2(\sE)=H^2+\alpha HF$ and charge $k=\alpha+c$ such that
\[
\Ext_X^1(\sE,\sE)=6(c-1+\alpha), \qquad \Ext_X^2(\sE,\sE)=\Ext_X^3(\sE,\sE)=0
\]
and such that $\sE$ is generically trivial on lines.

In particular there exists, inside the moduli space $\MI(H^2+\alpha HF)$ of $H$-instanton bundles with $c_2(\sE)=H^2+\alpha HF$, a generically smooth, irreducible component of dimension $6(c-1+\alpha)$.
\end{theoremint}
In particular, we prove the existence of $H$-instanton bundles for every admissible charge. Notice that in the general case we require that the coefficient of $H^2$ is equal to one. The strategy of the proof consists of applying subsequent elementary transformations along a line in the fiber and deforming the torsion free sheaf obtained to a locally free sheaf. The starting point of this process is a twist of the relative cotangent bundles, which is the unique, $\mu$-stable, (twisted) rank two Ulrich bundle on $S$ (see \cite[Theorem 3.7]{AHMP}), showing once more the strong connection between Ulrich and $H$-instanton bundles.

In the remaining part of the work we deal with the existence of $H$-instanton bundles with $c_2(\sE)=k_1H^2 +k_2HF$ and $k_1$ arbitrarly high. In order to do so, we need to restrict ourselves to some (infinite) subfamilies of rational normal scrolls.

\begin{theoremint}\label{teo2intro}
Let $S$ be a rational normal scroll of degree $c$ either of the form $S(a,a,b)$ with $b\leq a+1$ (resp. $S(a_0,a_1,a_2)$ with $a_0+a_1>a_2$ and even degree). Then for any $k_1\geq0$ and $k_2\geq 1-ck_1+a(k_1-1)$ (resp. $k_2\geq 1-ck_1+b_0(k_1+1)$) there exists a $\mu$-stable instanton bundle with $c_2(\sE)=k_1H^2+k_2HF$ which is generically trivial on lines in the fibers and such that
\[
\Ext_X^1(\sE,\sE)=2\left(c-5+2k_1(c+1)+3k_2\right), \qquad \Ext_X^2(\sE,\sE)=\Ext_X^3(\sE,\sE)=0.
\]
In particular there exists, inside the moduli space $\MI(k_1H^2+k_2HF)$ of instanton bundles with $c_2(\sE)=k_1H^2+k_2HF$, a generically smooth, irreducible component of dimension $2\left(c-5+2k_1(c+1)+3k_2\right)$.
\end{theoremint}

The proof of this theorem is obtained combining two different powerful techniques. The first step is to construct some suitable vector bundles using the Serre's correspondence between rank two vector bundles and one-dimensional subschemes on $S$. Indeed the main reason for the restrictions imposed on $S$ is to ensure the existence of suitable disjoint families of rational curves.

The second step is to use the strategy of the proof of Theorem \ref{teo1intro}. The crucial difference is that the base case of the subsequential elementary transformations is given by the vector bundles constructed via Serre's correspondence in step one.

As one can expect, the geometry of $S$ strongly depends on the invariants $a_i$. For this reason we treat separately the case $a_0=a_1$ (Section 5) and the case $a_0+a_1>a_2$ and even degree (Section 6). In the first case, the condition on the $a_i$'s guarantees the existence of a two-dimensional subvariety isomorphic to $\p\times\p$. In the second case, although the conditions seem a little curious, they are strongly related to the form of the general hyperplane section of the rational normal scroll.

{\bf{Acknowledgements}}: The authors want to thank C. Casagrande, G. Casnati, S. Marchesi and J. Pons-Llopis for the helpful discussions.

\section{$H$-instantons and classical instantons}

Let $\kk$ be an algebraically closed field of characteristic $0$. Given an integer $N$, set
$\PP^N$ for the projective space of hyperplanes through the origin of
$\kk^{N+1}$.
Let $X \subset \PP^N$
be a closed integral subscheme of dimension
$n$.
We assume throughout the paper that $X$ is non-degenerate, namely,
there is no hyperplane of $\PP^N$ that contains $X$.
The variety $X$ is equipped with the very ample line bundle $\sO_X(1)$
defined as the restriction of $\sO_{\PP^N}(1)$ via the embedding $X
\subset \PP^N$. We will write $H_X$ for the divisor
class of $\sO_X(1)$.

The coordinate ring $R$ of $\PP^N$ is the graded polynomial algebra in
$N+1$ variables with the standard grading, namely
$R=\kk[x_0,\ldots,x_N]$.
The homogeneous coordinate ring $\kk[X]$ is the graded algebra
$\kk[X]=R/I_X$, where $I_X$ is the homogeneous radical ideal of
polynomials vanishing on $X$.

The degree of $X$ is computed via the Hilbert polynomial of $I_X$. We will
be denoted it by $d$.

Given a coherent sheaf $\sE$ on $X$,
the $i$-th \textit{cohomology module} of $\sE$ is the $\kk[X]$-module:
\[
H^i_*(\sE)=\bigoplus_{k\in\ZZ}H^i(X,\sE\otimes\sO_X(kH_X)).
\]

For $i \ge 1$, the $\kk[X]$-modules $H^i_*(\sE)$ are artinian.
If $\sE$ is locally free, then $H^i_*(\sE)$ is finitely generated over $\kk[X]$.

\begin{definition}
A coherent sheaf $\sE$ on $X$  is called \textit{aCM}, standing for
\textit{Arithmetically Cohen-Macaulay}, if $\sE$ is locally
Cohen-Macaulay on $X$ and:
 \[
 H^i_*(\sE)=0, \qquad \mbox{$\forall i \in  \{1,\ldots,n-1\}$.}
 \]

 Equivalently, the
 minimal graded free resolution of the module of global sections
 $E=H^0_*(\sE)$, seen as $R$-module, has length $N-n$.
\end{definition}

The variety $X$ itself is said to be aCM if $X$ is projectively
 normal and $\sO_X$ is aCM. This is equivalent to ask that $\kk[X]$ is
 a graded Cohen-Macaulay ring, which in turn amounts to the fact
 that the minimal graded free resolution of
 $\kk[X]$ as $R$-module has length $N-n$.
 In this case, the line bundles $\sO_X(kH_X)$
 are aCM. Among aCM sheaves the Ulrich ones play a crucial role. 
 \begin{definition} Let $d$ be the degree of the embedded variety $X
   \subset \PP^N$.
   For $r>0$, a rank-$r$ aCM sheaf $\sE$ on $X$ is said to be
   \textit{Ulrich} if  $H^0(X,\sE(-H_X)) = 0$ and $h^0(X,\sE) =
   rd$.
 \end{definition}
 
 Let us recall the notion of (semi)stability. 
 \begin{definition}
 For each sheaf $\sE$ on $S$ the {\sl slope} of $\sE$ with respect to $H$ is the rational number $\mu(\sE):=c_1(\sE)H^2/\rk(\sE)$ and the reduced Hilbert
polynomial $p_{\sE}(t)$ of a bundle $\sE$ over $X$ is
$ p_{\sE}(t):=\chi(\sE(th))/\rk(\sE).$

We say that a vector bundle $\sE$ is {\sl $\mu$-stable} (resp. {\sl $\mu$-semistable}) with respect to $H$ if  $\mu( \sG) < \mu(\sE)$ (resp. $\mu(\sG) \le \mu(\sE)$) for each subsheaf $ \sG$ with $0<rk(\sG)<rk(\sE)$.

On the other hand, $\sE$ is said to be Gieseker semistable  with respect to $H$ if for all $ \sG$ as above one has
$$
p_{ \sG}(t) \leq  p_{\sE}(t),
$$
and Gieseker stable again if equality cannot hold in the above inequality.
\end{definition}

In this section we present some general results about $H$-instanton, relating them with the previous notions cited in the Introduction.

Let us recall that every instanton bundle $\sF$ with $c_2(\sF)=k$ on $\mathbb P^3$ is the cohomology of a linear monad of the form
\begin{equation}\label{monP}
0 \to \sO_{\mathbb P^3}(-1)^{\oplus k} \xrightarrow{\alpha}  \sO_{\mathbb P^3}^{\oplus 2k+2}
\xrightarrow{\beta} \sO_{\mathbb P^3}(1)^{\oplus k} \to 0.
\end{equation}
It is also possible to characterize instanton bundles on $\mathbb P^3$ as the cohomology of a second monad:
$$
0 \to \sO_{\mathbb P^3}(-1)^{\oplus k} \xrightarrow{\alpha}  \Omega_{\mathbb P^3}(1)^{\oplus k}
\xrightarrow{\beta} \sO_{\mathbb P^3}^{\oplus 2k-2} \to 0.$$

Conversely every rank two vector bundle defined as the cohomology of these monads is stable, has $c_1(\sF)=0$, $c_2(\sF)=k$ and $h^1(\pee3,\sF(-2))=0$, hence it is an instanton bundle. In light of this, we give the following definition:

\begin{definition}\label{hinstanton}
 Let  ($X$, $H_X$) be a polarized projective variety of dimension three and let $K_X$ be the canonical divisor.
 A rank  two vector bundle $\sE$ on $X$ is an $H$-instanton bundle of charge $k$ if the following properties hold:
\begin{itemize}
\item $c_1(\sE)=2H_X+K_X$ and $\deg(c_2(\sE))=k$;
\item $h^0(X,\sE)=0$ and $\sE$ is $\mu$-semistable;
\item $h^1(X,\sE(-H_X))=0$;
\end{itemize}
We call $E$ quasi $H$-instanton bundle if all the above conditions are satisfied except the $\mu$-semistability.

We call $E$ almost $H$-instanton bundle if all the above conditions are satisfied except $h^0(E)=0$.
\end{definition}

From the definition it follows that all the generators of the cohomology module $H_\ast^1(\sE)$ are concentrated in degree 0, i.e. we have the following Lemma (which is a slight generalization of \cite[Lemma 1]{Fa}).

\begin{lemma}\label{h1}
Let $\sE$ be an $H$-instanton bundle on a smooth polarized threefold $(X,\sO_X(H))$. Then 
$$
H^1(X,\sE(-tH))=0 \quad \text{for all $t\geq1$}.
$$
\end{lemma}
\begin{proof}
The proof goes by induction on $t$, with $t=1$ being the base case descending directly from Definition \ref{hinstanton}. Suppose $H^1(X,\sE(-kH))=0$ with $k>1$ and we show that $H^1(X,\sE(-(k+1)H))=0$. Let $S$ be a general hyperplane section of $X$ and consider the natural short exact sequence
\begin{equation}\label{hyperplane}
0\to \sO_X(-H)\to\sO_X\to \sO_S\to0.
\end{equation}
Now tensor \eqref{hyperplane} by $\sE(-kH)$. Since $\sE$ is $\mu$-semistable the same holds for its restriction to a general hyperplane section $S$ thanks to \cite[Theorem 3.1]{Mar}. In particular $H^0(\sE_S(-kH))=0$ for all $k<0$. Considering the long exact sequence induced in cohomology we obtain an injection $H^1(X,\sE(-(k+1)H))\subseteq H^1(X,\sE(-kH))$, but the latter is zero because of the inductive hypothesis, thus we obtain the thesis.
\end{proof}
In particular Lemma \ref{h1}, thanks to Serre's duality, implies that $H^2(X,\sE)=0$ and thus $\chi(\sE)=-h^1(X,\sE)$. Because of this it makes sense to speak about instanton of minimal charge, since they are the ones for which $\chi(\sE)=0$.

Now we deal with the relation between $H$-instanton and classical instanton bundles on $\pee3$.
\begin{itemize}

\item On $(\pee3, \Oo_{\pee3}(2))$ the $H$-instanton bundles coincide with the classical instanton bundles.

\item On $(\pee3,\sO_{\pee3}(1))$ if $\sF$ is a classical instanton bundle then $\sE\cong\sF(-1)$ is an $H$-instanton bundle, conversely if $\sE$ is an $H$-instanton bundle then $\sF\cong\sE(1)$ is a classical instanton bundle if and only if $\sE\ncong\Oo_{\pee3}(-1)\oplus\Oo_{\pee3}(-1)$. Indeed if $H=\Oo_{\pee3}(1)$,  then we get $c_1(\sE)=-2$ and $h^1(\pee3,\sE(-H))=0$. So we have that $\sF\cong\sE(1)$ is an instanton bundle in the usual sense with only difference being that $H^0(\pee3,\sF)$ could be non zero. By Ferrand's Remark \cite[Remark II.3.4]{OSS}, the only strictly $\mu$-semistable rank two vector bundle on $\mathbb P^3$ with $c_1(\sE)=-2$ and $h^1(\pee3,\sE(-1))=0$ is $\Oo_{\mathbb P^3}(-1)^{\oplus 2}$. Then we can conclude that the only strictly $\mu$-semistable case is the case of minimal charge $\Oo_{\mathbb P^3}(-1)^{\oplus 2}$ which is Ulrich up to twist. All the other $H$-instanton bundles coincide with usual ones twisted by $\sO_{\pp^3}(-1)$.

\item On $(\pee3,\sO_{\pp^3}(d))$ with $d\geq 2$ the situation is slightly different. If $\sE$ is an $H$-instanton bundle, $\sF=\sE(-d+2)$ is a classical instanton bundle with extra condition that $h^0(\pee3,\sF(d-2))=0$. Since $h^2(\pee3,\sF(d-2))=h^1(\pee3,\sF(-d-2))=0$, in order to have $h^0(\pee3,\sF(d-2))=0$ we must have $\chi(\sF(d-2))\leq 0$. By a Riemann Roch computation
$$\chi(F(d-2))=\frac{1}{3}d(d-1)(d+1)-d(d-2)^2-dc_2(F)\leq 0,$$
and thus we obtain $$c_2(F)\geq \frac{1}{3}(d-1)(d+1).$$
Hence, by \cite{HH} there exist rank two bundles $\sF$ with natural cohomology and in particular with $h^1(\sE)=-\chi(\sE)$ and such that $\sE$ has no global sections  if $$c_2(\sE)=k=c_2(\sF)+(d-2)^2\geq \frac{1}{3}(d-1)(d+1)+(d-2)^2.$$
We conclude that the $H$-instanton bundles $\sE$ have charge $k\geq \frac{1}{3}(d-1)(d+1)+(d-2)^2$ and arise from a monad
\begin{equation}\label{ver}
0\to\Oo_{\pp^3}(d-3)^{\oplus k-(d-2)^2}\to\Oo_{\pp^3}(d-2)^{\oplus 2(k-(d-2)^2)+2}\to\Oo_{\pp^3}(d-1)^{\oplus k-(d-2)^2}\to 0.
\end{equation}
If $d\not=3m$ the minimal charge is reached and all the rank two Ulrich bundles are $H$-instanton bundles tensored by $\Oo_{\pp^3}(d)$. In the case $d=3m$ the non-existence of Ulrich bundles was proved in \cite{ESW}.
\end{itemize}
In the case of Fano threefolds $X$ with index two or three and Picard number one, the definition of $H$-instanton bundle coincide with the one given in \cite{Fa}.

The case of index one is much more complicated. Considering the definition in Section 4 of \cite{Fa} an instanton bundle $\sF$ has $c_1(\sF)=-1$ and $h^1(X,\sF)=0$. So $\sE=\sF(1)$ is a $H$-instanton bundle if $h^0(X,\sF(1))=0$. Notice that $c_2(\sF)=c_2(\sE)$ and by Riemann--Roch
$$\chi(\sE)=\frac{1}{3}\bigl(2g_X-2-3c_2(E)+g_X+11\bigr),$$
where $g_X=-\frac{K_X^3}{2}+1$ is the genus of the Fano threefold.

Thanks to the extra condition $h^0(X,\sE)=0$ we get $\chi(\sE)=-h^1(X,\sE)\leq 0$, thus $c_2(\sE)\geq g_X+3$.

In \cite{Fa} the author proves that the moduli space of instanton bundles is non empty when $c_2(\sF)=k>\frac{g_X}{2}-1$ but the $H$-instanton bundles exist only when $k\geq g_X+3$. In particular for $g_X=12$ in \cite{Fa} is given a monad for any $k\geq 8$. The same monad tensored by $\Oo_X(1)$ produces also $H$-instanton bundles but only for $k\geq 15$. In this way the minimal charge correspond to a twist of an Ulrich bundle, thus also in the case of Fano threefold of index one and Picard one we have a strong analogy with the case of higher index described above. 

In the next sections our main characters will be rational normal scrolls of dimension three.
\section{Rational normal scrolls of dimension three}
Let $S=S(a_0, a_1, a_2)$ be a smooth rational normal scroll of dimension three over an algebraically closed field $k$ of characteristic 0. $S$ is the image of $\PP (\sD)$ via the morphism defined by $\Oo_{\PP (\sD)}(1)$, where $\sD \cong \oplus_{i=0}^2 \Oo_{\PP^1}(a_i)$ is a vector bundle of rank $3$ on $\PP^1$ with $0< a_0 \le  a_1 \le a_2$. Letting $\pi : \PP (\sD) \rightarrow \PP^1$ be the projection, we may denote by $H$ and $F$, the hyperplane section corresponding to $\Oo_{\PP(\sD)}(1)$ and the fibre corresponding to $\pi^*\Oo_{\PP^1}(1)$, respectively.

These varieties have a rather nice geometric description. In $\pp^N$ let us take three complementary linear subspaces $\Lambda_0$, $\Lambda_1$ and $\Lambda_2$ of dimension $a_0$, $a_1$ and $a_2$ respectively. In each $\Lambda_i$ consider a rational normal curve $C_i$ of degree $a_i$ and choose isomorphisms $\varphi_i:\p\to C_i$. Then we can describe the variety $S(a_0,a_1,a_2)$ as
\[
S(a_0,a_1,a_2)=\bigcup_{p\in\p}{\langle \varphi_0(p),\varphi_1(p),\varphi_2(p)\rangle}\subset \pp^N
\]
Observe that these varieties have degree $c=a_0+a_1+a_2$ in $\pp^{c+2}$, i.e. they are varieties of minimal degree.

We also have a complete description of the Chow ring of $S$. The Picard group $\textrm{Pic}(S)$ is given by $\textrm{Pic} (S)\cong \ZZ\langle H,F\rangle$, $A^2(X)\cong \ZZ \langle H^2,HF\rangle$  and we have the relations $H^3=c$, $H^2F$=1 and $F^2=0$. Moreover the canonical divisor is given by $\omega_S \cong-3H+(c-2)F$.

Let us recall some natural short exact sequences that we will use through the paper. Let us begin with the dual of the relative Euler exact sequence of $S$:
\begin{equation}\label{eq1}
0\to \Omega_{S|\PP^1}  \to \oplus_{i=0}^{2}\Oo_S(-H+a_iF) \to \Oo_S \to 0,
\end{equation}
and so we have that the relative canonical line bundle is given by $\omega_{S|\PP^1} \cong \Oo_S(-3H+cF)$. We will use also

\begin{equation}\label{eq2}
0\to \Oo_S(-3H+cF) \to \oplus_{i=0}^{2}\Oo_S(-2H+(c-a_i)F) \to \Omega_{S|\PP^1}  \to 0,
\end{equation}
and
\begin{equation}\label{eq3}
0\to \Oo_S(-2F) \to \Oo_S(-F)^{\oplus 2} \to \Oo_S  \to 0,
\end{equation}
Now we recall how to compute the cohomology of line bundles on $S$.
\begin{lemma}[\cite{EH}]\label{lem0}
For any $i=0,\ldots,n+1$, we have
\begin{itemize}
\item [(i)] $H^i(S, \Oo_S(aH+bF)) \cong H^i(\PP^1, \mathrm{Sym}^a\sD \otimes \Oo_{\PP^1}(b))$ if $a\ge 0$;
\item [(ii)] $H^i(S, \Oo_S(aH+bF)) =0$ if $-3<a<0$;
\item [(iii)] $H^i(S, \Oo_S(aH+bF)) \cong H^{3-i}(\PP^1, \mathrm{Sym}^{-a-3}\sD \otimes \Oo_{\PP^1}(c-b-2))$ if $a\le -3$.
\end{itemize}
\end{lemma}
Using Lemma \ref{lem0} it is not difficult to retrieve the following computations
\begin{remark}\label{rem0}
From the above Lemma we have
\begin{itemize}
\item [(i)] $H^i(S, \Omega_{S|\PP^1}(bF)) \cong H^{i-1}(S, \Oo_S(bF))$ for any $i$ and $b$ by (\ref{eq1}) tensored by $\Oo_S(bF)$;
\item [(ii)] $H^i(S, \Omega_{S|\PP^1}(H+bF))=0$ for any $i$ and $b$ by (\ref{eq2}) tensored by $\Oo_S(H+bF)$;
\item [(iii)] $H^i(S, \Omega_{S|\PP^1}(2H+bF)) \cong \oplus_{j=0}^{2}H^{i}(S, \Oo_S((c-a_j+b)F))$ for any $i$ and $b$ by (\ref{eq2}) tensored by $\Oo_S(2H+bF)$;
    \item [(iv)] $H^i(S, \Omega_{S|\PP^1}(-H+bF)^\vee) \cong \oplus_{j=0}^{2}H^{i}(S, \Oo_S((-a_j+b)F))$ for any $i$ and $b$ by the dual of (\ref{eq1}) tensored by $\Oo_S(-H+bF)$.
\end{itemize}
\end{remark}

In the particular case of rational normal scrolls of dimension three, Definition \ref{hinstanton} takes the following form.
\begin{definition}\label{scrollinsta}
A rank two $\mu$-semistable vector bundle $\sE$ on $S$ is called an $H$-instanton bundle of charge $k$ if and only if
\begin{itemize}
\item $c_1(\sE)=-H+(c-2)F$;
\item $h^0(S,\sE)=h^1(S,\sE(-H))=0$
\item $c_2(\sE)=k_1H^2+k_2HF$ with $k_1+k_2=k$.
\end{itemize}
\end{definition}
In what follows we will deal with instanton bundles with respect to the very ample line bundle $\sO_{\pp(\sD)}(1)$, thus we will often write instanton in place of $H$-instanton.

As a consequence of Definition \ref{scrollinsta}, since $\sE^\vee\cong\sE(H-(c-2)F)$, we have by Serre duality that
\begin{equation}\label{sd}
h^i(S,\sE(aH+bF))=h^{3-i}(S,\sE((-a-2)H-bF))
\end{equation}
and in particular $h^i(S,\sE(-H))=0$ for any $i$.

In order to check $\mu$-semistability of sheaves we will use the following criterion.
\begin{proposition}\cite[Theorem 3]{JMPS}\label{hoppe}
Let $\sE$ be a rank two vector bundle over a polycyclic variety $X$ and let $L$ be a polarization on $X$. $\sE$ is $\mu$-(semi)stable if and only if
\[
H^0(X,\sE\otimes \sO_X(B))=0
\]
for all $B \in Pic(X)$ such that $\delta_L(B) \underset{(<)}{\leq} -\mu_L(\sE)$, where $\delta_L(B)=\deg_L(\sO_X(B))$.
\end{proposition}
\begin{remark}\label{rem}
In our case $\mu_H(\sE)=(-H+(c-2)F)H^2/2=(-c+c-2)/2=-1$ and $\delta_H(aH+bF)=aH^3+bH^2F=ac+b$ so we get that $\sE$ is $\mu$-stable if and only if
$$H^0(S,\sE(aH+bF))=0$$ when $ac+b\leq 1$. In particular $H^0(S,\sE(F))=0$ and the condition $H^0(S,\sE)=0$ included in the definition of instanton bundle is implied by the $\mu$-semistability.
\end{remark}

In the following lemma we describe the cohomology of an instanton bundle on $S$. Indeed it is known that a general instanton bundle on $\pee3$ has natural cohomology (i.e. at each twist there is only one non-vanishing cohomology group) and this also partially translates to Fano varieties. In the case of rational normal scrolls we are able to retrieve this behaviour at least for some twist of the instanton bundle.

\begin{lemma}
\label{riv2}
Let $\sE$ be an instanton bundle on $S$. Then
\medbreak
\begin{enumerate}
\item[(i)]
$
\begin{sistema}
h^0(S,\sE(bF))=0 \ \text{for $b\leq 1$},\\
h^0(S,\sE(-H+bF))=0 \ \text{for $b\leq a_0+1$},\\
h^0(S,\sE(-2H+bF))=0 \ \text{for $b\leq 2a_0+1$},
\end{sistema}$
\medbreak
\item[(ii)]
$\begin{sistema}
h^3(S,\sE(-2H+bF))=0 \ \text{for $b\geq -1$},\\
h^3(S,\sE(-H+bF))=0 \ \text{for $b\geq -a_0-1$},\\
h^3(S,\sE(bF))=0 \ \text{for $b\geq -2a_0-1$}.
\end{sistema}$
\medbreak
\item[(iii)]
$
\begin{sistema}
h^1(S,\sE(-H+bF))=0 \ \text{for $b\leq 0$},\\
h^2(S,\sE(-H+bF))=0 \ \text{for $b\geq 0$}.
\end{sistema}
$
\medbreak
\item[(iv)] $h^2(S,\sE)=h^2(\sE(-F))=h^1(S,\sE(-2H))=h^1(S,\sE(-2H+F))=0$.
\end{enumerate}
\end{lemma}
\begin{proof}
Since $\sE$ is $\mu$-semistable $H^0(\sE(F))=0$. By sequence (\ref{eq3}) tensored by $\sE(F)$, we get $h^0(S,\sE(bF))=0$ for any $b\leq 1$. By sequence (\ref{eq2}) tensored by $\sE(2H +(-c+a_0+1)F)$, since $h^0(S,\oplus_{i=0}^{2}\sE((c-a_i-c+a_0+1)F))=0$ we get $h^0(S,\sE(-H+(a_0+1)F))=0$. So arguing as above we obtain  $h^0(S,\sE(-H+bF))=0$ for any $b\leq a_0+1$.
By sequence (\ref{eq2}) tensored by $\sE(H +(-c+2a_0+1)F)$, since $h^0(S,\oplus_{i=0}^{2}\sE(-2H+(c-a_i-c+2a_0+1)F))=0$ we get $h^0(S,\sE(-H+(2a_0+1)F))=0$. So arguing as above we obtain  $h^0(S,\sE(-2H+bF))=0$ for any $b\leq 2a_0+1$.
By Serre duality we obtain $(ii)$.

By sequence (\ref{eq3}) tensored by $\sE(-H+F)$, since $h^i(\sE(-H))=0$ for any $i$, we get $h^2(S,\sE(-H+F))=h^1(S,\sE(-H-F))=0$. By sequence (\ref{eq3}) tensored by $\sE(-H+bF)$, we obtain $(iii)$.

By sequence (\ref{eq1}) tensored by $\sE$, since $h^2(S,\oplus_{i=0}^{2}\sE(-H+a_iF))=0$ by $(iii)$, we get $h^2(\sE)=0$.
By sequence (\ref{eq1}) tensored by $\sE(-F)$, since $h^2(S,\oplus_{i=0}^{2}\sE(-H+(a_i-1)F))=0$ by $(iii)$, we get $h^2(S,\sE(-F))=0$. By (\ref{sd}) we obtain also $h^1(S,\sE(-2H))=h^1(S,\sE(-2H+F))=0$ and we have proved $(iv)$.


\end{proof}

\begin{lemma}
\label{riv4}
Let $\sE$ be an instanton bundle on $S$ with $a_2\leq 2$. Then
\begin{enumerate}
\item[(i)] $h^i(\sE\otimes\Omega_{S|\PP^1}(H-F))=h^i(\sE\otimes\Omega_{S|\PP^1}(H-2F))=h^i(\sE\otimes\Omega_{S|\PP^1}(H-(c-2)F))=0$ if $i\not= 1$.
\item[(ii)] $h^0(\sE\otimes\Omega_{S|\PP^1}(H-(c-1)F))=h^3(\sE\otimes\Omega_{S|\PP^1}(H-(c-1)F))=0$.

\end{enumerate}
\end{lemma}
\begin{proof}

By sequence (\ref{eq1}) tensored by $\sE(H-F)$, since $a_2\leq 2$, $h^0(\oplus_{i=0}^{2}\sE(a_i-1)F))=0$ by Lemma \ref{riv2} $(i)$, we get $h^0(\sE\otimes\Omega_{S|\PP^1}(H-F))=0$. This implies also $h^0(\sE\otimes\Omega_{S|\PP^1}(H-2F))=h^0(\sE\otimes\Omega_{S|\PP^1}(H-(c-2)F))=h^0(\sE\otimes\Omega_{S|\PP^1}(H-(c-1)F))=0$. Let us consider sequence (\ref{eq2}) tensored by $\sE(H-(c-1)F)$. Notice that  $h^3(\oplus_{i=0}^{2}\sE(-H+(c-a_i-c+1)F))=0$ by Lemma \ref{riv2} $(ii)$ since $a_2\leq 2$ implies $-a_2+1\geq -a_0-1$, so we get $h^3(\sE\otimes\Omega_{S|\PP^1}(H-(c-1)F))=0$. This also implies $h^3(\sE\otimes\Omega_{S|\PP^1}(H-F))=h^3(\sE\otimes\Omega_{S|\PP^1}(H-2F))=h^3(\sE\otimes\Omega_{S|\PP^1}(H-(c-2)F))=0$.

Tensoring sequence (\ref{eq2}) by $\sE(H-F)$, since $h^2(\oplus_{i=0}^{2}\sE(-H+(c-1-a_i)F))=0$ by Lemma \ref{riv2} $(iii)$, we get $h^2(\sE\otimes\Omega_{S|\PP^1}(H-F))=0$.By sequence (\ref{eq2}) tensored by $\sE(H-2F)$, since $h^2(\oplus_{i=0}^{2}\sE(-H+(c-2-a_i)F))=0$ by Lemma \ref{riv2} $(iii)$, we get $h^2(\sE\otimes\Omega_{S|\PP^1}(H-2F))=0$.

Finally let us tensor sequence (\ref{eq2})  by $\sE(H-(c-2)F)$. Since $a_2\leq 2$ we have $h^2(\oplus_{i=0}^{2}\sE(-H+(c-c+2-a_i)F))=0$ by Lemma \ref{riv2} $(iii)$, thus $h^2(\sE\otimes\Omega_{S|\PP^1}(H-F))=0$.
\end{proof}

We conclude this section by explicitly writing down the Riemann-Roch formula for a rank two vector bundle $\sE$ on $S$.
\begin{equation}\label{rrgeneral}
\chi(\sE)=2+\frac{1}{6}(c_1(\sE)^3-3c_1(\sE)c_2(\sE))-\frac{1}{4}K_S(c_1(\sE)^2-2c_2(\sE))+\frac{1}{12}c_1(\sE)(K_S^2+\Omega)
\end{equation}
where $K_S=-c_1(\sT_S)=-3H+(c-2)F$ and $\Omega=c_2(\sT_S)=3H^2+(6-2c)$. Indeed, we can compute $K_S$ and $\Omega$ using the short exact sequence
\begin{equation}\label{sesomega}
0\to \sO_S(-2F)\to \Omega_S\to \Omega_{S|{\mathbb{P}^1}}\to 0.
\end{equation}

 In particular, computing the Euler characteristic of an instanton bundle we obtain the following lemma:
\begin{lemma}
Let $\sE$ be an instanton bundle on $S$, then
\begin{equation}\label{rrinsta}
\chi(\sE)=1-Hc_2(\sE).
\end{equation}
\end{lemma}
Now, given a divisor $D$ and an instanton $\sE$ we recall the formula for $\chi(\sE(D))$:
\begin{equation}\label{rrinstatwisted}
\chi(\sE(D))=1-c_2(\sE)(D+H)+\frac{1}{6}D(2D^2+K_S^2+\Omega)+HD(K_S+2H+D).
\end{equation}

We will often use
\begin{equation}\label{0H}
\chi(\sE(bF))=-\chi(\sE(-2H-bF)=1-(c+b)k_1-k_2+b,
\end{equation}

\begin{equation}\label{+H}
\chi(\sE(H+bF))=2-(b+2c)k_1-2k_2+2c+4b.
\end{equation} and
\begin{equation}\label{-H}
\chi(\sE(-H+bF))=-k_1b.
\end{equation}
By (\ref{eq1}) tensored by $\sE(H-F)$ we get
\begin{equation}\label{Omeg1}
\chi(\sE\otimes\Omega_{S|\PP^1}(H-F))=\sum_{i=0}^2\chi(\sE((a_i-1)F))-\chi(\sE(H-F))=2-2k_1(c-1)-k_2-c
\end{equation}
By (\ref{eq1}) tensored by $\sE(H-2F)$ we get

\begin{equation}\label{Omeg2}
\chi(\sE\otimes\Omega_{S|\PP^1}(H-2F))=\sum_{i=0}^2\chi(\sE((a_i-2)F))-\chi(\sE(H-2F))=3-k_2-k_1(2c-4)-k_2-c
\end{equation}
By (\ref{eq2}) tensored by $\sE(H-(c-2)F)$ we get

\begin{align}\label{Omeg3}
\chi(\sE\otimes\Omega_{S|\PP^1}(H+bF))& =\sum_{i=0}^2\chi(\sE((-H+(c-a_i+b)F))-\chi(\sE(-2H+(c+b)F)) \notag \\
&=\sum_{i=0}^2\chi(\sE((-H+(c-a_i+b)F))+\chi(\sE(-(c+b)F)) \notag \\
&=1-b-c-(2c+2b)k_1-k_2. \notag
\end{align}

 \section{Monads on three-dimensional rational normal scrolls}

We will use the following version (explained in \cite[Theorem 2.5]{AHMP}) of the Beilinson spectral sequence (see also  \cite[Corollary 3.3.2]{RU}, \cite[Section 2.7.3]{GO} and  \cite[Theorem 2.1.14]{BO}.) Let us briefly recall the following definitions.

Let $D^b(X)$ be the the bounded derived category of coherent sheaves on a smooth projective variety $X$.
\begin{itemize}
\item An object $E \in D^b(X)$ is called \emph{exceptional} if $\Ext^\bullet(E,E) = \mathbb C$.
\item A set of exceptional objects $\langle E_1, \ldots, E_n\rangle$ on  $X$ is called an \emph{exceptional collection} if $\Ext^\bullet(E_i,E_j) = 0$ for $i > j$.
\item An exceptional collection is \emph{full} when $\Ext^\bullet(E_i,A) = 0$ for all $i$ implies $A = 0$, or equivalently when $\Ext^\bullet(A, E_i) = 0$ for all $i$ also implies $A = 0$.
\item A full exceptional collection $\langle E_1, \ldots, E_n\rangle$ is \emph{strong} if $\Ext^k(E_i,E_j) = 0$ for $i<j$ and $k>0$.
\end{itemize}

\begin{theorem}\label{use}
Let $X$ be a smooth projective variety with a full exceptional collection
$\langle E_0, \ldots, E_n\rangle$
where $E_i=\mathcal E_i^*[-k_i]$ with each $\mathcal E_i$ a vector bundle and $(k_0, \ldots, k_n)\in \mathbb Z^{\oplus n+1}$ such that there exists a sequence $\langle F_n=\mathcal F_n, \ldots, F_0=\mathcal F_0\rangle$ of vector bundles satisfying
\begin{equation}\label{order}
\mathrm{\Ext}^k(E_i,F_j)=H^{k+k_i}(X, \mathcal E_i\otimes \mathcal F_j) =  \left\{
\begin{array}{cc}
\mathbb C & \textrm{\quad if $i=j=k$} \\
0 & \textrm{\quad otherwise}
\end{array}
\right.
\end{equation}

Then for any coherent sheaf $A$ on $X$ there is a spectral sequence in the square $-n\leq p\leq 0$, $0\leq q\leq n$  with the $E_1$-term
\[
E_1^{p,q} = \mathrm{\Ext}^{q}(E_{-p},A) \otimes F_{-p}=
H^{q+k_{-p}}(X,\mathcal E_{-p}\otimes A) \otimes \mathcal F_{-p}
\]
which is functorial in $A$ and converges to
\begin{equation}
E_{\infty}^{p,q}= \left\{
\begin{array}{cc}
A & \textrm{\quad if $p+q=0$} \\
0 & \textrm{\quad otherwise.}
\end{array}
\right.
\end{equation}
\end{theorem}

\begin{remark}\label{rembeil}
It is possible to state a stronger version of the Beilinson's theorem (see \cite{Be} for $\mathbb{P}^N$ and \cite{AO3} for the projectivized of a direct sum of line bundles over $\mathbb{P}^N$). Let us consider a rational normal scroll $S$ and let $A$ be a coherent sheaf on $X$. Let $\langle E_0,\dots,E_n\rangle$ be a full exceptional collection and $\langle F_n,\dots,F_0 \rangle$ its right dual collection. Using the notation of Theorem \ref{use}, if $\langle F_n,\dots,F_0\rangle$ is a strong exceptional collection of vector bundles, then there exists a complex $L^\bullet$ such that
\begin{enumerate}
\item $H^k(X,L^\bullet)=
\begin{cases}
A \ & \text{if $k=0$},\\
0 \ & \text{otherwise}.
\end{cases}$
\item $L^k=\underset{k=p+q}{\bigoplus}H^{q+k_{-p}}(X,A\otimes \sE_{-p})\otimes \sF_{-p}$ with $0\le q \le n$ and $-n\le p \le 0$.
\end{enumerate}
\end{remark}

\begin{definition}\label{def:mutation}
Let $E$ be an exceptional object in $D^b(X)$.
Then there are functors $\mathbb L_{E}$ and $\mathbb R_{E}$ fitting in distinguished triangles
$$
\mathbb L_{E}(T) 		\to	 \Ext^\bullet(E,T) \otimes E 	\to	 T 		 \to	 \mathbb L_{E}(T)[1]
$$
$$
\mathbb R_{E}(T)[-1]	 \to 	 T 		 \to	 \Ext^\bullet(T,E)^* \otimes E	 \to	 \mathbb R_{E}(T)	
$$
The functors $\mathbb L_{E}$ and $\mathbb R_{E}$ are called respectively the \emph{left} and \emph{right mutation functor}.
\end{definition}

Now we construct the full exceptional collections that we will use in the next theorems. Let us consider the full exceptional collection $\langle E_7, \ldots, E_0\rangle$ given in \cite{AHMP} Example 3.6:

\begin{gather}\label{col}
\{ \Oo_S(-2H)[-2], \Oo_S(-2H+F)[-2], \Oo_S(-H)[-1], \\
\Oo_S(-H+F)[-1], \Oo_S(-F), \Oo_S \} \notag
\end{gather}

The associated full exceptional collection $\langle F_7=\mathcal F_7, \ldots, F_0=\mathcal F_0\rangle$ of Theorem \ref{use} is

\begin{gather}\label{cold}
\{ \Oo_S(-H+(c-2)F), \Oo_S(-H+(c-1)F), \Omega_{S|\PP^1}^{1}(H-2F), \\
\Omega_{S|\PP^1}^{1}(H-F), \Oo_S(-F), \Oo_S \} \notag
\end{gather}

If we tensor \eqref{col} by $\sO((c-3)F)$ we obtain

\begin{gather}\label{col0}
\{ \Oo_S(-2H+(c-3)F)[-2], \Oo_S(-2H+(c-2)F)[-2], \Oo_S(-H+(c-3)F)[-1], \\
\Oo_S(-H+(c-2)F)[-1], \Oo_S((c-4)F), \Oo_S((c-3)F) \} \notag
\end{gather}

Thanks to a right mutation of the pair $\{\Oo_S(-2H+(c-3)F), \Oo_S(-2H+(c-2)F)\}$ we get:

\begin{gather}\label{col00}
\{ \Oo_S(-2H+(c-2)F)[-2], \Oo_S(-2H+(c-1)F)[-2], \Oo_S(-H+(c-3)F)[-1], \\
\Oo_S(-H+(c-2)F)[-1], \Oo_S((c-4)F), \Oo_S((c-3)F) \} \notag
\end{gather}
The associated left mutation is

\begin{gather}\label{cold0}
\{ \Oo_S(-H), \Oo_S(-H+F), \Omega_{S|\PP^1}^{1}(H-(c-1)F), \\
\Omega_{S|\PP^1}^{1}(H-(c-2)F), \Oo_S(-(c-2)F), \Oo_S(-(c-3)F) \} \notag
\end{gather}

From \eqref{col} with $t\geq 0$ right mutations of the pair  $\{\Oo_S(-F), \Oo_S\}$ and a right mutation of the pair  $\{\Oo_S(-2H), \Oo_S(-2H+F)\}$ we obtain:

\begin{gather}\label{colt}
\{ \Oo_S(-2H+F)[-2], \Oo_S(-2H+2F)[-2], \Oo_S(-H)[-1], \\
\Oo_S(-H+F)[-1], \Oo_S((t-1)F), \Oo_S(tF) \} \notag
\end{gather}
for any $t\geq 0$.\\
The associated left mutations of the pair $\{ \Oo_S(-F), \Oo_S \}$ and  $\{\Oo_S(-H+(c-2)F), \Oo_S(-H+(c-1)F)\}$ in \eqref{cold} give

\begin{gather}\label{coltd}
\{ \Oo_S(-H+(c-3)F), \Oo_S(-H+(c-2)F), \Omega_{S|\PP^1}^{1}(H-2F), \\
\Omega_{S|\PP^1}^{1}(H-F), \Oo_S(-(t+1)F), \Oo_S(-tF) \} \notag
\end{gather}
for any $t\geq 0$.\\

From \eqref{colt} with a right mutation of the pair  $\{\Oo_S(-2H)[-2], \Oo_S(-2H+F)[-2]\}$ we obtain:

\begin{gather}\label{colt2}
\{ \Oo_S(-2H+F)[-2], \Oo_S(-2H+2F)[-2], \Oo_S(-H)[-1], \\
\Oo_S(-H+F)[-1], \Oo_S((t-1)F), \Oo_S(tF) \} \notag
\end{gather}
for any $t\geq 0$.

The associated left mutation of the pair $\{ \Oo_S(-H+(c-2)F), \Oo_S(-H+(c-1)F) \}$ in \eqref{cold} gives

\begin{gather}\label{coltd2}
\{ \Oo_S(-H+(c-3)F), \Oo_S(-H+(c-2)F), \Omega_{S|\PP^1}^{1}(H-2F), \\
\Omega_{S|\PP^1}^{1}(H-F), \Oo_S(-(t+1)F), \Oo_S(-tF) \} \notag
\end{gather}
for any $t\geq 0$.

If we tensor (\ref{colt2}) by $\Oo_S((c-3)F)$ and \eqref{coltd2} by $\Oo_S(-(c-3)F)$ we obtain:
\begin{gather}\label{col*}
\{ \Oo_S(-2H+(c-2)F)[-2], \Oo_S(-2H+(c-1)F)[-2], \Oo_S(-H+(c-3)F)[-1], \\
\Oo_S(-H+(c-2)F)[-1], \Oo_S((c-4+t)F), \Oo_S((c-3+t)F) \} \notag
\end{gather}
and
\begin{gather}\label{cold*}
\{ \Oo_S(-H), \Oo_S(-H+F), \Omega_{S|\PP^1}^{1}(H-(c-1)F), \\
\Omega_{S|\PP^1}^{1}(H-(c-2)F), \Oo_S(-(c-2+t)F), \Oo_S(-(c-3+t)F) \} \notag
\end{gather}

From (\ref{col}) with $c-2$ right mutations of the pair  $\{\Oo_S(-2H)[-2], \Oo_S(-2H+F)[-2]\}$ we obtain:

\begin{gather}\label{col3}
\{ \Oo_S(-2H+(c-2)F)[-2], \Oo_S(-2H+(c-1)F)[-2], \Oo_S(-H)[-1], \\
\Oo_S(-H+F)[-1], \Oo_S(-F), \Oo_S \} \notag
\end{gather}

The associated left mutations of the pair $\{ \Oo_S(-H+(c-2)F), \Oo_S(-H+(c-1)F) \}$ in \eqref{cold} give
\begin{gather}\label{cold3}
\{ \Oo_S(-H), \Oo_S(-H+F), \Omega_{S|\PP^1}^{1}(H-2F), \\
\Omega_{S|\PP^1}^{1}(H-F), \Oo_S(-F), \Oo_S \} \notag
\end{gather}

From \eqref{col3} with $c-2$  left mutations of the pair  $\{\Oo_S(-F), \Oo_S\}$ we obtain:

\begin{gather}\label{col4}
\{ \Oo_S(-2H+(c-2)F)[-2], \Oo_S(-2H+(c-1)F)[-2], \Oo_S(-H)[-1], \\
\Oo_S(-H+F)[-1], \Oo_S(-(c-2)F), \Oo_S(-(c-3)F) \} \notag
\end{gather}

The associated right mutations of the pair $\{ \Oo_S(-F), \Oo_S \}$ in (\ref{cold3}) give

\begin{gather}\label{cold4}
\{ \Oo_S(-H), \Oo_S(-H+F), \Omega_{S|\PP^1}^{1}(H-2F), \\
\Omega_{S|\PP^1}^{1}(H-F), \Oo_S((c-4)F), \Oo_S((c-3)F) \} \notag
\end{gather}
\begin{remark}
The full exceptional collections (\ref{colt}), (\ref{colt2}), (\ref{col*}) and (\ref{cold4}) are strong.
\end{remark}
In what follows we give some monadic description of instanton bundles on $S$. At first we will consider the rational normal scrolls $S(a_0,a_1,a_2)$ with $a_2 \leq 2$ and then we will deal with the general case.

\begin{theorem}\label{mon1}
Let $a_2\leq 2$ and let $\sE$ be an instanton bundle on $S(a_0,a_1,a_2)$ with $c_2(\sE)=k_1H^2+k_2HF$. Then $\sE$ is the homology of a monad $$0 \to \sA  \xrightarrow{\alpha}  \sB
\xrightarrow{\beta} \sC \to 0,$$ where $\sA$ arises from the short exact sequence
$$0\longrightarrow\Oo_S(-2H+F)^{\oplus (c-3)k_1}\longrightarrow
\begin{matrix}
\Oo_S(-2H+2F)^{\oplus (c-2)k_1}\\
\oplus \\
\Oo_S(-H)^{\oplus (2c-4)k_1+k_2+c-3}
\end{matrix}
\longrightarrow \sA\longrightarrow0,$$
and the other terms of the monad are given by
$$\sB=\Oo_S(-H+F)^{\oplus (2c-2)k_1+k_2+c-2}\oplus\Oo_S(-F)^{\oplus (c-1)k_1+k_2}$$ and $$\sC=\sO_S^{\oplus ck_1+k_2-1}.$$
\end{theorem}

\begin{proof}
 We consider the Beilinson type spectral sequence associated to an instanton bundle $\sE$ and identify the members of the graded sheaf associated to the induced filtration as the sheaves mentioned in the statement of Theorem \ref{use}.
We consider the full exceptional collection $\langle E_7, \ldots, E_0\rangle$  given in (\ref{colt}) and the full exceptional collection $\langle F_7, \ldots, F_0\rangle$ given in (\ref{coltd}) with $t=0$.
 So we get the following table

{\small
  \begin{center}\begin{tabular}{|c|c|c|c|c|c|}
\hline
$\Oo_S(-2H+F)$& $\Oo_S(-2H+2F)$& $\Oo_S(-H)$&$\Oo_S(-H+F)$& $\Oo_S(-F)$ & $\Oo_S$\\
\hline
\hline
$H^3$		&	 $H^3$	& $*$		& $*$			& $*$			& $*$	\\
$H^2$		&$ H^2$		 	& $H^3$	& $H^3$	& $*$			& $*$	\\
		$H^1$		& $H^1$	&	$H^2$	& $H^2$		& $H^3$		& $H^3$	\\
	$H^0$&  $H^0$&	$H^1$	& $H^1$		& $H^2$		& $H^2$	\\
$*$		&	 $*$  		& $H^0$		& $H^0$	&	 $H^1$		& $H^1$	\\
$*$		&	 $*$		&	$*$	& $*$		& $H^0$	&	 $H^0$ \\
\hline
\hline
$\Oo_S(-H+(c-3)F)$  & $\Oo_S(-H+(c-2)F)$  & $\Omega(H-2F)$ & $\Omega(H-F)$ & $\Oo_S(-F)$ &  $\Oo_S$\\
\hline
\end{tabular}
\end{center}
}
Each entry of the table denotes the dimension of $H^i(\sE\otimes \sB)$ where $\sB$ is the correspondent bundle in the last row. In particular $\ast$ denotes the terms that are trivially zero.

By Lemma \ref{riv2}, Lemma \ref{riv4} and Riemann-Roch formula we have:
\begin{itemize}
\item $\chi(\sE)=-h^1(S,\sE)=1-ck_1-k_2$.
\item $\chi(\sE(-F))=-h^1(S,\sE(-F))=-(c-1)k_1-k_2$.
\item $\chi(\sE\otimes\Omega_{S|\PP^1}(H-F))=-h^1(S,\sE\otimes\Omega_{S|\PP^1}(H-F))=2-c-(2c-2)k_1-k_2.$
\item $\chi(\sE\otimes\Omega_{S|\PP^1}(H-2F))=-h^1(S,\sE\otimes\Omega_{S|\PP^1}(H-2F))=3-c-(2c-4)k_1-k_2.$
\item $\chi(\sE(-H+(c-2)F))=-h^1(S,\sE(-H+(c-2)F))=-(c-2)k_1$.
\item $\chi(\sE(-H+(c-3)F))=-h^1(S,\sE(-H+(c-3)F))=-(c-3)k_1$.
\end{itemize}
So we obtain the following table
{\footnotesize
  \begin{center}
\addtolength{\leftskip}{-2cm}
\addtolength{\rightskip}{-2cm}
\begin{tabular}{|c|c|c|c|c|c|}
\hline
$\Oo_S(-2H+F)$& $\Oo_S(-2H+2F)$& $\Oo_S(-H)$&$\Oo_S(-H+F)$& $\Oo_S(-F)$ & $\Oo_S$\\
\hline
\hline
$0$		&	 $0$	& $*$		& $*$			& $*$			& $*$	\\
$0$		&$ 0$		 	& $0$	& $0$	& $*$			& $*$	\\
		$(c-3)k_1$		& $(c-2)k_1$	&	$0$	& $0$		& $0$		& $0$	\\
	$0$&  $0$&	$c+(2c-4)k_1+k_2-3$	& $c+(2c-2)k_1+k_2-2$		& $0$		& $0$	\\
$*$		&	 $*$  		& $0$		& $0$	&	 $(c-1)k_1+k_2$		& $ck_1+k_2-1$	\\
$*$		&	 $*$		&	$*$	& $*$		& $0$	&	 $0$ \\
\hline
\hline
$\Oo_S(-H+(c-3)F)$  & $\Oo_S(-H+(c-2)F)$  & $\Omega(H-2F)$ & $\Omega(H-F)$ & $\Oo_S(-F)$ &  $\Oo_S$\\
\hline
\end{tabular}
\end{center}
}
So we obtain the claimed monad.
\end{proof}
Now we construct a second type of monad for an instanton bundle.
\begin{theorem}\label{mon2}
Let $a_2\leq 2$ and let $\sE$ be an instanton bundle on $S(a_0,a_1,a_2)$ with $c_2(\sE)=k_1H^2+k_2HF$. Let us write $\alpha=h^2(S,\sE(-(c-3)F)$,  $\beta=h^2(\sE(-(c-2)F)$, $\delta=h^2(S,\sE\otimes\Omega(H-(c-1)F)$.  Then $\sE$ is the homology of a monad

$$0 \to \sA  \xrightarrow{\alpha}  \sB
\xrightarrow{\beta} \sC \to 0,$$ where $$\sA=\sO(-2H+(c-1)F)^{\oplus k_1}\oplus\sO(-H+(c-3)F)^{\oplus 2k_1+k_2+\delta},$$
\[
\sB=\Oo_S(-H+(c-3)F)^{\oplus \delta}\oplus\Oo_S(-H+(c-2)F)^{\oplus 4k_1+k_2+1}\oplus \sO((c-4)F)^{\oplus 2k_1+k_2+c-3+\beta}
\]
and $\sC$ arises from the short exact sequence
\[
0\longrightarrow \sC\longrightarrow \begin{matrix}

\sO((c-4)F)^{\oplus \beta}\oplus\sO_S((c-3)F)^{\oplus 3k_1+k_2+c-2+\alpha}
\end{matrix}
\longrightarrow \sO_S((c-3)F)^{\oplus \alpha}\longrightarrow0.
\]
\end{theorem}

\begin{proof}
 We consider the Beilinson type spectral sequence associated to an instanton bundle $E$ and identify the members of the graded sheaf associated to the induced filtration as the sheaves mentioned in the statement of Theorem \ref{use}.
We consider the full exceptional collection $\langle E_7, \ldots, E_0\rangle$  given in (\ref{col00}) with $t=0$ and the full exceptional collection $\langle F_7, \ldots, F_0\rangle$ given in (\ref{cold0}). We do not consider the column associated to $E_7$ since $H^i(\sE(-H))=0$ for any $i$.
 So we get the following table

{\small
  \begin{center}\begin{tabular}{|c|c|c|c|c|}
\hline
 $\sO((c-1)F-2H)$ &  $\sO((c-3)F-H)$ &  $\sO((c-2)F-H)$ &  $\sO((c-4)F)$ &  $\sO((c-3)F)$\\
\hline
\hline
 $H^3$		 	& $*$		 	& $*$		& $*$ &  $*$	\\
$ H^2$	 	& $H^3$	 	& $H^3$	&	 $*$			& $*$	\\
		 $H^1$	&	 $H^2$	&	 $H^2$	&	 $H^3$		& $H^3$	\\
	$H^0$	& $H^1$		& $H^1$		& $H^2$	& $H^2$	\\
 $*$  			& $H^0$		& $H^0$		& $H^1$	&	 $H^1$	\\
 $*$			& $*$			& $*$		& $H^0$		& $H^0$ \\
\hline
\hline
 $\sO(-H+F)$  & $\Omega(H-(c-1)F)$ & $\Omega(H-(c-2)F)$  & $\sO((2-c)F)$ & $\sO((3-c)F)$\\
\hline
\end{tabular}
\end{center}
}
By Lemma \ref{riv2}, Lemma \ref{riv4} and Riemann-Roch formula we obtain:
\begin{itemize}
\item $\chi(\sE((3-c)F))=-h^1(S,\sE((3-c)F))+\alpha=-(3k_1+k_2+c-4)$.
\item $\chi(\sE((2-c)F))=-h^1(S,\sE((2-c)F))+\beta=-(2k_1+k_2+c-3)$.
\item $\chi(\sE(-H+F))=-h^1(S,\sE(-H+F))=-k_1$.
\item $\chi(\sE\otimes\Omega_{S|\PP^1}(H-(c-2)F))=-h^1(S,\sE\otimes\Omega_{S|\PP^1}(H-(c-2)F))+\gamma=4k_1+k_2+1+\gamma$.
\item $\chi(\sE\otimes\Omega_{S|\PP^1}(H-(c-1)F))=-h^1(S,\sE\otimes\Omega_{S|\PP^1}(H-(c-1)F))+\delta=2k_1+k_2+\delta$.
\end{itemize}

So we have the following table

{\footnotesize
  \begin{center}\begin{tabular}{|c|c|c|c|c|}
\hline
 $\sO((c-1)F-2H)$  & $\sO((c-3)F-H)$  & $\sO((c-2)F-H)$  & $\sO((c-4)F)$  & $\sO((c-3)F)$\\
\hline
\hline
 $0$		 	& $*$		 	& $*$			& $*$  & $*$	\\
$ 0$	 	& $0$	 	& $0$	&	 $*$			& $*$	\\
		 $k_1$		& $\delta$		& $\gamma$		& $0$		& $0$	\\
	$0$ 	& $2k_1+k_2+\delta$		& $4k_1+k_2+1+\gamma$		& $\beta$		& $\alpha$	\\
 $*$  		&	$0$	&	$0$	&	$2k_1+k_2+c-3+\beta$	&	 $3k_1+k_2+c-4+\alpha$	\\
 $*$		&	 $*$		&	 $*$		& $0$	& $0$ \\
\hline
\hline
 $\sO(-H+F)$  & $\Omega(H-(c-1)F)$  &$\Omega(H-(c-2)F)$  & $\sO((2-c)F)$  &$\sO((3-c)F)$\\
\hline
\end{tabular}
\end{center}
}

So we obtain the claimed monads.

\end{proof}

\begin{remark}
In this remark we discuss the shape of the previous monads in some particular cases.

When $c=3$ the monads of Theorem \ref{mon1} and Theorem \ref{mon2} coincide. In this case $\alpha=\beta=\gamma=\delta=0$ by Lemma \ref{riv2} and the monad becomes

$$0\longrightarrow
\begin{matrix}
\sO(-2H+2F)^{\oplus k_1} \\
\oplus \\
\sO(-H)^{\oplus 2k_1+k_2}
\end{matrix}
\longrightarrow
\begin{matrix}
\Oo_S(-H+F)^{\oplus 4k_1+k_2+1}\\
\oplus \\
\sO(-F)^{\oplus 2k_1+k_2}
\end{matrix}
\longrightarrow \sO_S^{\oplus 3k_1+k_2-1}\longrightarrow 0.$$

When $c=4$, then  $\alpha=\gamma=0$ by Lemma \ref{riv2} and the monad of Theorem \ref{mon2}  becomes

$$0\longrightarrow
\begin{matrix}
\sO(-2H+3F)^{\oplus k_1} \\
\oplus \\
\sO(-H+F)^{\oplus 2k_1+k_2}
\end{matrix}
\longrightarrow
\begin{matrix}
\Oo_S(-H+F)^{\oplus \delta} \\
\oplus\\
\Oo_S(-H+2F)^{\oplus 4k_1+k_2+1}\\
\oplus \\
\sO^{\oplus 2k_1+k_2+1}
\end{matrix}
\longrightarrow
\begin{matrix}
\sO^{\oplus \beta} \\
\oplus \\
\sO_S(F)^{\oplus 3k_1+k_2-1}
\end{matrix}\longrightarrow 0.$$
\end{remark}
Now we are ready to deal with the general case.
\begin{theorem}\label{mon3}
Let $\sE$ be an instanton bundle on $S$ with $c_2(\sE)=k_1H^2+k_2HF$. Let us write $\alpha=h^2(\sE(-(c-3)F)$, $\beta=h^2(\sE(-(c-2)F)=h^1(\sE(-2H+(c-1)F)$, $\theta=h^1(\sE(-2H+(c-1)F)$. Then $\sE$ is the homology of a monad

\begin{equation}\label{monadegenerale}
0 \to
 \sA  \xrightarrow{\alpha}  \sB
\xrightarrow{\beta} \sC \to 0,
\end{equation}
where $\sA$ arises from the short exact sequence
\[
0 \longrightarrow
\sO_S(-H)^{\oplus \beta} \longrightarrow
\begin{matrix}
\Oo_S(-H)^{\oplus c-3+2k_1+k_2+\beta}\\
\oplus \\
\sO(-H+F)^{\oplus \theta}
\end{matrix}
\longrightarrow \sA\to 0,
\]
$\sC$ arises from the short exact sequence
\[
0
\longrightarrow
\sC\
\longrightarrow
\begin{matrix}
\sO((c-4)F)^{\oplus \beta} \\
\oplus\\
\sO_S((c-3)F)^{\oplus 3k_1+k_2+c-4+\alpha}
\end{matrix}
\longrightarrow
\sO_S((c-3)F)^{\oplus \alpha}\to 0
\]
and $\sB$ is given by
$$\sB=\Oo_S(-H+F)^{\oplus c-2+k_1+k_2+\theta}\oplus\Omega(H-F)^{\oplus k_1}\oplus \sO((c-4)F)^{\oplus 2k_1+k_2+c-3+\beta}.$$
\end{theorem}

\begin{proof}
 We consider the Beilinson type spectral sequence associated to an instanton bundle $\sE$ and identify the members of the graded sheaf associated to the induced filtration as the sheaves mentioned in the statement of Theorem \ref{use}.
We consider the full exceptional collection $\langle E_7, \ldots, E_0\rangle$  given in (\ref{col4}) and the full exceptional collection $\langle F_7, \ldots, F_0\rangle$ given in (\ref{cold4}).
 So we get the following table

{\footnotesize
  \begin{center}\begin{tabular}{|c|c|c|c|c|c|}
\hline
$\sO(-H)$  & $\sO(-H+F)$  & $\Omega(H-2F)$ & $\Omega(H-F)$ & $\sO((c-4)F)$ &  $\sO((c-3)F)$\\
\hline
\hline
$H^3$		&	 $H^3$	& $*$		& $*$			& $*$			& $*$	\\
$H^2$		&$ H^2$		 	& $0$	& $H^3$	& $*$			& $*$	\\
		$H^1$		& $H^1$	&	$0$	& $H^2$		& $H^3$		& $H^3$	\\
	$H^0$&  $H^0$&	$0$	& $H^1$		& $H^2$		& $H^2$	\\
$*$		&	 $*$  		& $0$		& $H^0$	&	 $H^1$		& $H^1$	\\
$*$		&	 $*$		&	$*$	& $*$		& $H^0$	&	 $H^0$ \\
\hline
\hline
$\Oo_S(-2H+(c-2)F)$  & $\Oo_S(-2H+(c-1)F)$  &$\Oo_S(-H)$ &$\Oo_S(-H+F)$  & $\Oo_S((2-c)F)$ &  $\Oo_S((3-c)F)$\\
\hline
\end{tabular}
\end{center}
}

By Lemma \ref{riv2} and Riemann-Roch formula we obtain:
\begin{itemize}
\item $\chi(\sE((3-c)F))=-h^1(S,\sE((3-c)F))+\alpha=-(3k_1+k_2+c-4)$.
\item $\chi(\sE((2-c)F))=-h^1(S,\sE((2-c)-F))+\beta=-(2k_1+k_2+c-3)$.
\item $\chi(\sE(-H+F))=-h^1(S,\sE(-H+F))=-k_1$.
\item $\chi(\sE(-2H+(c-1)F)=-\chi(\sE((1-c)F))=h^2(S,\sE(-2H+(c-1)F)-\theta=c-2+k_1+k_2$.
\item $\chi(\sE(-2H+(c-2)F)=-\chi(\sE((2-c)F))=h^2(S,\sE(-2H+(c-2)F)-\beta=c-3+2k_1+k_2$.
\end{itemize}
So we have the following table (we omit the third column since it is made of all zeros).
{\footnotesize
\begin{center}\begin{tabular}{|c|c|c|c|c|}
\hline
$\sO(-H)$ &  $\sO(-H+F)$   & $\Omega(H-F)$  & $\sO((c-4)F)$  & $\sO((c-3)F)$\\
\hline
\hline
$0$		 	& $0$		& 	 $*$			& $*$  & $*$	\\
$ c-3+2k_1+k_2+\beta$	 	& $c-2+k_1+k_2+\theta$	 	& $0$		& $*$			& $*$	\\
		 $\beta$		& $\theta$		& $0$		& $0$		& $0$	\\
	$0$&	 $0$		& $k_1$		& $\beta$		& $\alpha$	\\
$*$  			& $*$		& $0$		& $2k_1+k_2+c-3+\beta$		& $3k_1+k_2+c-4+\alpha$	\\
$*$		&	 $*$			& $*$		& $0$		& $0$ \\
\hline
\hline
$\Oo_S(-2H+(c-2)F)$  & $\Oo_S(-2H+(c-1)F)$   &$\Oo_S(-H+F)$   & $\Oo_S((2-c)F)$  & $\Oo_S((3-c)F)$\\
\hline
\end{tabular}
\end{center}}
So we obtain the claimed monad.
\end{proof}


\begin{remark}The maps in the exact sequences defining $\sA$ and $\sC$ are pullbacks from $\mathbb P^1$. So we have $$\sA=\oplus_{i=1}^z\Oo_S(-H+l_iF)$$ where $z=c-3+2k_1+k_2+\theta$ and $l_i\geq 0$ for any $i$ and
$$\sC=\oplus_{j=1}^y\Oo_S(l_jF)$$ where $y=3k_1+k_2+c-2+\beta$ and $l_j\leq c-4$ for any $j$. When $c=3$ then $\alpha=\beta=0$ and the monad becomes

\[
0 \longrightarrow
\begin{matrix}
\Oo_S(-H)^{\oplus 2k_1+k_2} \\
\oplus \\
\sO(-H+F)^{\oplus \theta}
\end{matrix}
\longrightarrow
\begin{matrix}
\Oo_S(-H+F)^{\oplus 1+k_1+k_2+\theta} \\
\oplus \\
\Omega(H-F)^{\oplus k_1} \\
\oplus \\
\sO(-F)^{\oplus 2k_1+k_2}
\end{matrix}
\longrightarrow
\sO_S^{\oplus 3k_1+k_2-1}\to 0.
\]
\end{remark}

It is worth to notice that Theorem \ref{mon3} imposes some necessary conditions on the second Chern class of an instanton bundle. Indeed we obtain that if $\sE$ is an instanton bundle with $c_2=k_1H^2+k_2HF$ then $k_1\geq 0$. However by Equation \eqref{rrinsta} and Lemma \ref{riv2} we obtain $h^1(S,\sE)=c_2(\sE)H-1=ck_1+k_2-1\geq 0$, thus  $k_2\geq 1-ck_1$. Recall that since $h^2(S,\sE)=0$ then minimal charge instanton bundles are actually Ulrich. The only rank two Ulrich bundles satisfy $k_1=0$ and $k_2=1$, or $k_1=1$ and $k_2=1-c$ (see \cite[Remark 3.10]{AHMP} and \cite[Proposition 3.3]{MR}. If we impose the conditions $k_1=1$ and $k_2=1-c$ in Theorem \ref{mon3} we get $\sA=\sC=0$ and $\sE=\sB=\Omega_{S|\p}(H-F)$ which is the unique $\mu$-stable rank two Ulrich bundle $\Omega_{S|\p}(H-F)$ twisted by $\sO_S(-H)$. The other minimal charge case is represented by the strictly $\mu$-semistable rank two Ulrich bundles (i.e. an extension of two Ulrich line bundles). In particular we can conclude that if $k_1 \geq 2$ then $k_2 >1-ck_1$.

In the following proposition we show that instanton bundles with $k_1=0$ actually behave like the strictly $\mu$-semistable rank two Ulrich bundles (i.e. $k_1=0$ and $k_2=1$)

\begin{proposition}
Let $\sE$ be an instanton bundle on $S$ with $c_2(\sE)=k_1H^2+k_2HF$, then $\sE$ is an extension of two line bundles if and only if $k_1=0$. In this case $\sE$ arises from the following extension
\begin{equation}\label{ext}
0 \to \sO(-k_2F)  \rightarrow \sE\rightarrow \sO(-H+(c+k_2-2)F) \to 0.
\end{equation}
\end{proposition}

\begin{proof}
Let us start by showing that any  instanton bundle with $c_2(\sE)=k_2HF$ is an extension of the form \eqref{ext}. Let $t_0$ be the minimal integer such that $h^0(S,\sE(t_0F))>0$. Let $s$ be a section $s \in H^0(S,\sE(t_0F))$, thus we have the following short exact sequence
\begin{equation}\label{serresemistable}
0\to \sO_S \to \sE(t_0F) \to \sI_{C|S}(-H+(c-2+2t_0)F)\to 0,
\end{equation}
where $C$ is a curve in the class $c_2(\sE(t_0F))=(k_2-t_0)HF$. Now we show that $t_0=k_2$. Since $\sE$ is an instanton bundle, it is the cohomology of a monad of the form \eqref{monadegenerale}. Twisting $\eqref{monadegenerale}$ by $\sO_S(t_0F)$ and taking the long exact sequence induced in cohomology by its display, we obtain
\[
h^0(S,\sE(t_0F))=t_0-k_2+1.
\]
Since this $h^0(S,\sE(t_0F))=1$ by \eqref{serresemistable}, we obtain $t_0=k_2$ and $C=\emptyset$, thus $\sE$ fits into \eqref{ext}.

Conversely, suppose $\sE$ is an instanton bundle which is realized as an extension of two line bundles
\[
0\to \sO_S(A) \to \sE \to \sO_S(B)\to0.
\]
Then a straightforward computation of the Chern classes shows that necessarily $c_2(\sE)=k_2HF$.
\end{proof}

In the next sections we will prove the existence of H-instanton bundles on $S$ with any given admissible second Chern class, given some restriction on $S$ itself.

We conclude this section with a remark on the $\pee1 \times \pee2$ case.

\begin{remark}
In this remark we will point out that there are vector bundles realized as the cohomology of a monad as in Theorem \ref{mon3}, which satisfy all the instantonic condition but the $\mu$-semistability. Let us consider the rational normal scroll $S=\pee1\times\pee2$ and let us take a line in the fiber, i.e. a line $l$ in the class $|HF|$. It is straightforward to see that $\sN_{l/S} \cong \sO_\p(1) \oplus \sO_\p$. In particular $\det \sN_{l/S}$ is the restriction to $l$ of any line bundle of the form $\sO_S(H+bF)$, thus we can suppose that $\det \sN_{l/S}\cong \sO_S(H-5F)\otimes \sO_l$. Since $h^i(\sO_S(-H+5F))=0$ for each $i$, then by Serre's Correspondence (see Proposition \ref{HScorr}) there exists a unique rank two vector bundle realized as an extension
\begin{equation}\label{serrenosemistable}
0\to \sO_S \to \sF \to \sI_{l|S}(H-5F) \to 0.
\end{equation}
Let us denote by $\sE$ the vector bundle $\sE\cong \sF(-H+3F)$. By a computation of the Chern classes we obtain $c_1(\sE)=-H+F$ and $c_2(\sE)=3HF$. Furthermore twisting \eqref{serrenosemistable} by $\sO_S(-H+3F)$ and taking the cohomology we obtain $h^0(\sE)=0$. To check the instantonic condition it is enough to twist \eqref{serrenosemistable} by $\sO_S(-2H+3F)$ and since  both $\sI_{l|S}(-H-2F)$ and $\sO_S(-2H+3F)$ are acyclic we obtain $h^1(S,\sE(-H))=0$. Thus $\sE$ has the same invariants and the same cohomology table of an instanton bundle on $S$, thus it can be realized as a cohomology of a monad as in Theorem \ref{mon3}. However, by construction, $h^0(S,\sE(H-3F))=1$, thus by Proposition \ref{hoppe} we conclude that $\sE$ cannot be $\mu$-semistable.
\end{remark}

\section{Construction of instanton bundles with increasing $c_2$}
In this section we describe an iterative process to construct instanton bundles with increasing charge. Let us begin by considering the base case of this process. Recall that a minimal charge instanton bundle is actually an Ulrich bundle (up to twist). The only rank two stable Ulrich bundle on $S(a_0,a_1,a_2)$ is given by $\Omega_{S|\pp^1}(2H-F)$ (see \cite[Theorem 3.8]{AHMP}) and the correspondent instanton is $\Omega_{S|\pp^1}(H-F)$.

\begin{lemma}\label{restriction}
$\Omega_{S|\pp^1}(H-F) \otimes \sO_L \cong \sO_\p \oplus \sO_\p(-1)$ for the generic line $L \in |HF|$.
\end{lemma}
\begin{proof}
By Grothendieck's theorem $\Omega_{S|\pp^1}(H-F) \otimes \sO_L\cong \sO_{\p}(a)\oplus \sO_\p (b)$. Take the short exact sequence \eqref{eq1} and tensor it by $\sO_S(H-F)\otimes \sO_L$. Computing the first Chern classes we obtain $a+b=-1$. But we have an injective map
\[
0\to \Omega_{S|\pp^1}(H-F) \otimes \sO_L \to \sO_\p^{\oplus 3}
\]
thus we obtain that $a$ and $b$ are both less or equal than $0$, so the only choice is $\Omega_{S|\pp^1}(H-F) \otimes \sO_L\cong \sO_{\p}\oplus \sO_\p (-1)$.
\end{proof}

\begin{construction}\label{deformation}
In this construction, starting with a $\mu$-stable instanton bundle $\sE$ with charge $c_2(\sE)=k_1H^2+k_2HF$. Suppose $\sE_L \cong \sO_\p \oplus \sO_\p(-1)$ for the generic line in $|HF|$ and $\Ext^2(\sE,\sE)\cong\Ext^3(\sE,\sE)=0$, then we will construct a stable instanton bundle $\sF$ with $c_2(\sE)=k_1H^2+(k_2+1)HF$, generically isomorphic to $ \sO_\p \oplus \sO_\p(-1)$ on lines in $|HF|$ and such that $\Ext^2(\sF,\sF)\cong \Ext^3(\sF,\sF)=0$.

\vspace{0.3cm}
{\textbf{Step 1:}}Defininig a sheaf with increasing $c_2$.

By hypothesis there exists a surjective map $\sE \xrightarrow{\phi} \sO_L \to 0$. Let $\sE_\phi$ be the kernel of such morphism, i.e.
\begin{equation}\label{defses}
0 \to \sE_\phi \to \sE \xrightarrow{\phi}\sO_L\to 0.
\end{equation}
$\sE_\phi$ is a torsion free sheaf, which is not locally free. We show that $\sE_\phi$ satisfies all the other instantonic conditions.

$\bullet \ H^0(S,\sE_\phi)=0$ and $H^1(S,\sE(-H))=0$.

Take the sequence induced in cohomology by \eqref{defses}. Then $h^0(S,\sE_\phi)\leq h^0(S,\sE)=0$. To show the second vanishing it is enough to twist \eqref{defses} by $\sO(-H)$ and observe that $h^i(S,\sE(-H))=h^i(S,\sO_L(-1))=0$ for all $i$.

$\bullet \ c_1(\sE_\phi)=c_1(\sE)$ and $c_2(\sE_\phi)=c_2(\sE)+HF$.

In order to compute the Chern classes it is enough to observe that $c_1(\sO_L)=0$, $c_2(\sO_L)=-HF$ and $c_3(\sO_L)=1$.

$\bullet \ \Ext^2(\sE_\phi,\sE_\phi)=0$ and $\Ext^3(\sE_\phi,\sE_\phi)=0$.

Apply the covariant functor $\Hom (\sE,-)$ to \eqref{defses}. We have $\Ext^i(\sE,\sO_L)\cong H^i(\p,\sE_L^\vee)\cong H^i(\p,\sO_\p \oplus \sO_\p(1))$, which vanishes for all $i>0$. Thus we obtain $\Ext^2(\sE,\sE_\phi)\cong \Ext^2(\sE,\sE)=0$. Now we apply the contravariant functor $\Hom(-,\sE_\phi)$ to \eqref{defses} we obtain the following injective map
\[
0 \to \Ext^2(\sE_\phi,\sE_\phi)\to \Ext^3(\sO_L,\sE_\phi).
\]
We show that $\Ext^3(\sO_L,\sE_\phi)=0$. By duality we have $\Ext^3(\sO_L,\sE_\phi)\cong \Hom(\sE_\phi,\sO_\p(-3))^\vee$. Since a line in $|HF|$ is a locally complete intersection then $\mathcal{E}xt^i(\sO_L,\sO_L)\cong \Lambda ^i(\sN_{L|S})$ and in particular $\mathcal{E}xt^1(\sO_L, \sO_L(-3))\cong \sO_\p(-2)\oplus\sO_\p(-3)$. Now apply $\mathcal{H}om(-,\sO_L(-3))$ to \eqref{defses}, thus we have
\[
0 \to \mathcal{H}om(\sO_L,\sO_L(-3)) \to \mathcal{H}om(\sE,\sO_L(-3)) \to\mathcal{H}om(\sE_\phi,\sO_L(-3))\to \sO_L(-2)\oplus \sO_L(-3)\to 0
\]
which becomes
\[
0 \to \sO_L(-3) \to \sO_L(-2)\oplus \sO_L(-3) \to \mathcal{H}om(\sE_\phi,\sO_L(-3)) \to \sO_L(-2)\oplus \sO_L(-3)\to 0,
\]
thus we obtain $\Hom(\sE_\phi,\sO_L(-3)) \cong H^0(S,\mathcal{H}om(\sE_\phi,\sO_L(-3)))=0$ and finally $\Ext^2(\sE_\phi,\sE_\phi)=0$.

To show that $\Ext^3(\sE_\phi,\sE_\phi)=0$ observe that the $\mu$-stability of $\sE$ implies the $\mu$-stability of $\sE_\phi$. Since $\sE_\phi$ is simple, by \cite[Lemma 2.1]{CCGM} we obtain $\Ext^3(\sE_\phi,\sE_\phi)=0$.
In particular $\sE_\phi$ correspond to a smooth point in the moduli space of $\mu$-stable sheaves $M_S(2,c_1(\sE),c_2(\sE)+HF)$. In order to compute its dimension we use the relation
\[
\chi(\sE_\phi,\sE_\phi)=\chi(\sE,\sE)-\chi(\sE,\sO_L)-\chi(\sO_L,\sE)+\chi(\sO_L,\sO_L).
\]
Using the exact sequences as in the previous arguments it is straightforward to compute $\chi(\sE,\sO_L)=1$ and $\chi(\sO_L,\sE)=3$ In order to compute $\chi(\sO_L,\sO_L)$ apply $\Hom(\sO_L,-)$ to the exact sequence
\[
0\to \sO_S(-H-F) \to \sO_S(-H) \oplus \sO_S(-F) \to \sO_S \to \sO_L \to 0.
\]
Thus we have $\chi(\sO_L,\sO_L)=\chi(\sO_L,\sO_S)-\chi(\sO_L,\sO_S(-H)\oplus \sO_S(-F))=0$. In particular $\chi(\sE_\phi,\sE_\phi)=\chi(E,E)-4$ and
\[
\dim \Ext^1(\sE_\phi,\sE_\phi)=\dim \Ext^1(\sE,\sE)+4.
\]
Moreover $\sE_\phi \otimes \sO_{L'}\cong \sO_\p \oplus \sO_\p(-1)$ for every line $L'$ non intersecting $L$. But every line which does not lie in the same fiber as $L$ is disjoint from $L$, and since having this splitting type is an open property, then $\sE_\phi \otimes \sO_{L}\cong \sO_\p \oplus \sO_\p(-1)$ for the generic line in $|HF|$.

\vspace{0.3cm}
{\textbf{Step 2:}} Deforming $\sE_\phi$ to a locally free sheaf.

Let us take $\sF$ a general deformation of $\sE_\phi$ in $M_S(2,c_1(\sE),c_2(\sE)+HF)$  and $L'$ a general deformation of $L$. We want to show that $\sF$ is a locally free sheaf, thus it is an instanton bundle.

The strategy is to show that if $\sF$ is not locally free then it would sit in a short exact sequence of the form
\begin{equation}\label{badses}
0\to\sF\to\sE'\to \sO_{L'} \to 0,
\end{equation}
where $\sE'$ is a deformation of $\sE$. But this is not possible since $\sF$ lies in a component of the moduli space of dimension $\dim \Ext^1(\sE,\sE)+4$, but the vector bundles lying in \eqref{badses} describe a family of dimension $\dim\Ext^1(\sE,\sE)+2$, thus $\sF$ must be locally free.

Suppose therefore that $\sF$ is not locally free and consider the natural bidual sequence
\begin{equation}\label{biduale}
0 \to \sF \to \sF^{\vee\vee}\to T\to 0
\end{equation}

Let us denote by $Y$ the support of $T$. Since we supposed $\sF$ not locally free, we have that $Y\neq \emptyset$. Furthermore $T$ is supported in codimension at least two. We say that $Y$ has pure dimension one.

In fact twisting \eqref{biduale} by $\sO_S(-H)$ we observe that if $H^0(S,\sF^{\vee \vee}(-H))\neq 0$ then a nonzero global section of $\sF^{\vee \vee}$ will induce via pull-back a subsheaf $K$ of $\sF$ with $c_1(K)=h$, which is not possible since $\sF$ is stable. So we have $H^0(S,\sF^{\vee \vee}(-H))\cong H^1(S,\sF(-H))=0$ which implies $H^0(S,T(-H))=0$. In particular $Y$ has no embedded points, i.e. is pure of dimension one. We want to show that $Y$ is actually a line.

Let $H$ be a general hyperplane section which does not intersect the points where $F^{\vee\vee}$ is not locally free. Tensor \eqref{defses} by $\sO_H$. Since $H$ is general the sequence remains exact and $\sO_{L\cap H}$ is supported at one point, which represent the point where $\sE_\phi \otimes\sO_H$ fails to be reflexive (in this case also locally free). $\sF$ is a deformation of $\sE_\phi$ and because of the choice of $H$, restricting \eqref{biduale} to $H$ does not affect the exactness of the short exact sequence. Moreover $T_H$ is supported on points where $\sF_H$ is not reflexive. Since being reflexive is an open condition, by semicontinuity $T_H$ is supported at most at one point. But $Y$ cannot be empty and is purely one dimensional, thus $Y\cap H$ consists of one point and $Y$ must be a line $L$. Furthermore by semicontinuity $T$ is of generic rank one and we have $c_2(T)H=-1$ (see \cite[Example 15.3.1]{fulton}).

Now choose $t<<0$. By Serre's vanishing and taking the cohomology of \eqref{biduale} we have that $h^1(Y,T(tH))\leq h^2(S,\sF(tH))$.  Observe that $d=c_3(F^{\vee \vee})$ and $c_2(T)$ are invariant for twists. By computing the Chern classes using \eqref{biduale} we have $c_3(T)=d$ and $$c_3(T(tH))=d-c_2(T)c_1(\sF(tH))= d-c_2(T)\bigl((2t-1)H+(d-2)F\bigr).$$
By Riemann-Roch we have $\chi(T(tH))=\frac{1}{2}d+(t+1)$, thus
\[
-\frac{1}{2}d-(t+1)=h^1(Y,T(th))\leq h^2(S,\sF(tH))\leq -(t+1)
\]
and in particular $d\leq0$. However $d=c_3(F^{\vee \vee})$ and $d\geq 0$ since $\sF^{\vee\vee}$ is reflexive, thus we must have $d=0$.
This implies that $\sF^{\vee \vee}$ is locally free. Moreover $c_2(T)=-HF$, so we get $c_2(\sF^{\vee \vee})=aH^2+(b+1)HF$, which implies that $\sF^{\vee \vee}$ has the same Chern classes as $\sE$ and is $\mu$-stable because of the $\mu$-stability of $\sF$. Therefore, $\sF^{\vee \vee}$ is a flat deformation of $\sE$ and it is also stable, so $\sF^{\vee \vee}$ lies in a neighborhood of $\sE$ in $M_S(2,c_1(\sE),c_2(\sE))$.

\end{construction}
By applying recursively Construction \ref{deformation} to the unique $\mu$-stable rank two (twisted) Ulrich bundle $\Omega_{S|\p}(H-F)$ we obtain the following result.
\begin{theorem}
Let $S=S(a_0,a_1,a_2)$ be a rational normal scroll of degree $c$. Then for every $\alpha\geq 1-c$ there exists a $\mu$-stable instanton bundle $\sE$ with $c_2(\sE)=H^2+\alpha HF$ such that
\[
\Ext_X^1(\sE,\sE)=6(c-1+\alpha), \qquad \Ext_X^2(\sE,\sE)=\Ext_X^3(\sE,\sE)=0
\]
and such that $E$ is generically trivial on lines.

In particular there exists, inside the moduli space $\MI(H^2+\alpha HF)$ of $\mu$-stable instanton bundles with $c_2(\sE)=H^2+\alpha HF$, a generically smooth, irreducible component of dimension $6(c-1+\alpha)$.
\end{theorem}
\begin{proof}
The only thing left to compute is the dimension of $\Ext_X^1(\sE,\sE)$. It follows directly from Riemann-Roch recalling that $c_1(\sE\otimes\sE^\vee)=c_3(\sE\otimes \sE^\vee)=0$ and $c_2(\sE\otimes\sE^\vee)=4c_2(\sE)-c_1^2(\sE)$.
\end{proof}

\begin{remark}
Construction \ref{deformation} it is a useful tool to show the existence of instanton bundles of every charge by starting with a known one. In particular we are able to increase the degree of the $c_2(E)$ by one at each step by increasing the coefficient of the class $HF$. However, as we will see in the next sections, it could also allow us to prove the existence of instanton bundle for a much larger set of values of $c_2(\sE)$ by starting Construction \ref{deformation} with instanton bundle for which the $H^2$ term in the second Chern class has coefficient arbitrarly high.
\end{remark}

\section{The case $a_0=a_1$}
In this section we consider rational normal scrolls $S(a_0,a_1,a_2)$ of the form $S(a,a,b)$ with $a\leq b$ and $2a+b=c$.
In this cases $S(a,a,b)\cong \pp(\sO_\p(a) \oplus \sO_\p(a) \oplus \sO_\p(b))\to \p$.

Let us consider the subbundle $\sO_\p(a) \oplus \sO_\p(a)\subset \sO_\p(a) \oplus \sO_\p(a) \oplus \sO_\p(b)$. Then $Q\cong \pp(\sO_\p(a)\oplus \sO_\p(a))$ is naturally a subvariety of $S(a,a,b)$ which is isomorphic to $\p\times\p$. By \cite[Proposition 9.13]{EH2} the class of $Q$ inside $A^1(S)$ is given by $H-bF$ and $\mathcal{N}_{Q|S} \cong \sO_S(H-bF)_{|_Q}$.

$Q$ is isomorphic to $\pp(\sO_\p\oplus \sO_\p)$ embedded via $C_0+af$ where $C_0$ and $f$ represent the two rulings of $Q$. In particular $C_0$ is a rational normal curve represented in $A^2(S)$ by the class $H^2-(a+b)HF$.  Let us denote by $\Gamma(S)$ the Hilbert scheme of such curves. Notice that $\Gamma(S)$ is isomorphic to $\p$ since $|C_0|$ represents one of the two rulings of $Q$.

These curves will be the main characters in the construction of instantons. Let us recall the Serre's correspondence between two-codimensional subschemes of $S$ and rank two vector bundles supported on $S$ itself.

\begin{proposition}\label{HScorr}
Let $X$ be a projective variety and let $Z\subseteq X$ be a local complete intersection subscheme of codimension 2. If $\det (\sN_{Z|X} )\cong \sO_Z \otimes \sL$ for some $\sL\in \Pic(S)$ such that $h^2(X,\sL^\vee)=0$, then there exists a vector bundle $F$ of rank two on $X$ such that:
\begin{itemize}
\item[(i)] $\det(F)\cong \sL$;
\item[(ii)] $F$ has a section $s$ such that $Z$ coincides with the zero locus $(s)_0$ of $s$.
\end{itemize}
Moreover, if $H^1(X,\sL^\vee)=0$, the above two conditions determine $F$ up to isomorphism.
\end{proposition}
\begin{proof}
See \cite{arrondo}.
\end{proof}

Let us begin by showing how the curves $H^2-(a+b)HF$ are related to the unique stable rank two Ulrich bundle on $S$, namely $\Omega_{S|\p}(2H-F)$.
\begin{lemma}\label{sectionulrich}
$\Omega_{S|\p}(2H-(a+b)F)$ has a section vanishing in pure codimension two over a rational normal curve in the class $H^2-(a+b)HF$.
\end{lemma}

\begin{proof}
From sequence \eqref{eq2} tensored by $\sO_S(-(a+b-1)F))$ we obtain $h^0(S,\Omega_{S|\p}(2H-(a+b)F))=1$. Take a generic section $s\in H^0(S,\Omega_{S|\p}(2H-(a+b)F))$, and let us denote by $Y=(s)_0$ locus where $s$ vanishes. Observe that $Y$ has pure codimension two because $H^0(S,\Omega_{S|\p}(2H-(a+b)F)\otimes \sO_S(D))=0$ for any effective divisor $D$. Thus we have that $\Omega_{S|\p}(2H-(a+b)F)$ fits in the short exact sequence
\[
0\to\sO_S\to\Omega_{S|\p}(2H-(a+b)F)\to \sI_Y(H-bF)\to0.
\]
To compute the class of $Y$ it is enough to compute $c_2(\Omega_{S|\p}(2H-bF))$.
\end{proof}
Now we inductively construct instanton bundles starting from the (twisted) Ulrich bundle $\Omega_{S|\p}(H-F)$.
\begin{construction}\label{disjointrational}
Let us consider $C_1,C_2,\dots,C_\alpha$ disjoint rational normal curves from the family $H^2-(a+b)HF$. First of all observe that we can choose such curves so that they are all pairwise disjoint, since they are curves in one of the ruling of $\pp(\sO_\p(a)\oplus \sO_\p(a))\subset S(a,a,b)$. Let us denote by Y the one-dimensional scheme
\begin{equation}\label{0locus2}
Y=\bigcup_{j=i}^{\alpha}{C_i}
\end{equation}
We claim that that $\det \mathcal{N}_{Y/S}\cong \sO_S(H-(a+b)F)\otimes \sO_Y$. We can verify such an isomorphism component by component. As a consequence of Lemma 	\ref{sectionulrich} we have that $\sN_{{C_i}|S}\cong \Omega_{S|\p}(2H-(a+b)F)\otimes\sO_{C_i}$, thus restricting \eqref{eq2} to $C_i$ we obtain $\sN_{C_i|S}\cong\sO_\p(a-b) \oplus \sO_\p$. In particular
\[
\det \sN_{C_i/S} \otimes \sO_{C_i} \cong \sO_S(H-bF) \otimes \sO_{C_i}.
\]
 i.e. the determinant of the normal bundle of $Y$ is extendable on $S$. Since $h^i(S,\sO_S(-H+bF))=0$ for all $i$, it follows from Proposition \ref{HScorr} that there exists a unique vector bundle (up to isomorphism) $\sF$ on $S$ with a section $s$ vanishing exactly along $Y$ and with $c_1(\sF)=H-bF$ and $c_2(\sF)=Y$.
Thus $\sE=\sF(-H+(a+b-1)F)$ has $c_1(\sE)=-H+(c-2)F$ and $c_2(\sE)=\alpha H^2-(\alpha(a+b)+a-1)HF$ and it fits into
\begin{equation}\label{serreinsta}
0\rightarrow \sO_X(-H+(a+b-1)F) \rightarrow \sE \rightarrow \sI_{Y|S}\bigl((a-1)F\bigr) \to 0.
\end{equation}
\end{construction}
\begin{proposition}\label{serreexistence2}
Let $\sE$ be a vector bundle with $c_1(\sE)=-H+(c-2)F$, $c_2(\sE)=\alpha H^2-(\alpha(a+b)+a-1)HF$ with $\alpha\geq 1$. Suppose $\sE(H-(a+b-1)F)$ has a section vanishing along $Y$ as in Construction \ref{disjointrational}, i.e. $E$ fits into
\begin{equation}\label{serreinsta2}
0\rightarrow \sO_S(-H+(a+b-1)F) \rightarrow \sE \rightarrow \sI_{Y|S}\bigl((a-1)F\bigr) \to 0.
\end{equation}
Then $\sE$ is a $\mu$-stable instanton bundle with charge $k=a(\alpha-1)+1$ such that $\sE_L\cong \sO_\p \oplus \sO_\p(-1)$ for the generic line L in the fibers.
\end{proposition}
\begin{proof}
First we check that $\sE$ has no sections. Indeed taking the cohomology of \eqref{serreinsta2} we obtain $H^0(S,\sE)\cong H^0(Y,\sI_{Y|S}\bigl((a-1)F\bigr))$. However $Y$ is not contained in any surface in the class $(a-1)F$. Indeed $Y$ consists of disjoint union of rational normal curves intersecting transversely each fiber $F$. But the surface $(a-1)F$ must be the disjoint union of $a-1$ fibers. Thus $Y \not\subset (a-1)F$ and $H^0(Y,\sI_Y((a-1)F))=0$.

Now we prove that $\sE$ satisfies the instantonic condition. If we tensor \eqref{serreinsta} by $\sO_S(-H)$ we obtain $H^i(S,\sE(-H))\cong H^i(Y,\sI_{Y|S}\bigl(-H+(a-1)F\bigr))=0$ for all $i$.

To prove the $\mu$-semistability and the behaviour of $\sE$ when restricted to lines in the family $HF$ we proceed by induction on $\alpha$.

If $\alpha=1$ then by Lemma \ref{sectionulrich} the vector bundle $\sE$ is a twist of the rank two Ulrich bundle $\Omega_{S|\p}(2H-F)$, which is $\mu$-stable (see \cite{AHMP}). Now suppose that for $\alpha=t$ Construction \ref{disjointrational} yields a $\mu$-stable bundle and we prove it for $\alpha=t+1$. Now let $Y'=\bigcup_{i=1}^{t}{C_i}$, so that $Y=Y' \cup C_{t+1}$ where $C_i$ are rational normal curves in the class $H^2-(a+b)HF$ and they are pairwise disjoint. Let us denote by $\sE'$ the vector bundle associated to $Y'$ via Construction \ref{disjointrational}. Thus we have the two natural short exact sequences
\begin{gather}
0 \to \sO_{Y'} \to \sO_Y \to \sO_{C_{t+1}}\to 0  \label{natural1}\\
0 \to \sI_{Y|S} \to \sI_{Y'|S} \to \sO_{C_{t+1}} \to 0. \label{natural2}
\end{gather}
In particular there exists a surjective map $\phi$
\[
\sI_{Y'|S}\bigl((a-1)F\bigr)\xrightarrow{\phi}\sO_{C_{t+1}}\bigl((a-1)F\bigr) \to 0.
\]
Now let us compose this map with the surjective map in \eqref{serreinsta}
\begin{equation}
\sE' \xrightarrow{\psi} \sI_{Y'|S}\bigl((a-1)F\bigr)\xrightarrow{\phi}\sO_{C_{t+1}}\bigl((a-1)F\bigr) \to 0.
\end{equation}
Since both $\phi$ and $\psi$ are surjective by the Snake's lemma we have a short exact sequence
\[
0 \to \ker \psi \to \ker (\phi  \psi) \to \ker \phi \to 0,
\]
which becomes
\begin{equation}
0\to  \sO_S(-H+(a+b-1)F) \to  \ker (\phi  \psi) \to \sI_{Y|S}\bigl((a-1)F\bigr)\to0.
\end{equation}
In particular $\sE=\ker (\phi  \psi)$ is the unique vector bundle associated to $Y$ via Construction \ref{disjointrational}. Thus we have the following exact sequence
\begin{equation}\label{finalses}
0 \longrightarrow \sE \longrightarrow \sE' \xrightarrow{\phi \psi} \sO_{C_{t+1}}\bigl((a-1)F\bigr) \longrightarrow 0.
\end{equation}
If there exists a line bundle destabilizing $\sE$, it would also destibilize $\sE'$ but this is not possible since $\sE'$ is $\mu$-stable by inductive hypothesis, thus $\sE$ is $\mu$-stable. Moreover, suppose by induction on $\alpha$ that $\sE'\otimes \sO_L\cong \sO_\p \oplus \sO_\p(-1)$ for the generic line $L \in |HF|$. We checked the case $\alpha =1$ in Construction \ref{deformation}. Since every curve $C_j$ intersects the fiber trasversally, it si possible to choose $L'\in |HF|$ such that $L'\cap Y =\emptyset$. For such line, restricting \eqref{finalses} to $L'$ we obtain $\sE\otimes \sO_{L'}\cong \sO_\p \oplus \sO_\p(-1)$, but having this splitting type is an open condition, thus it must hold for the generic line.
\end{proof}
In the next Proposition we show that in the special case $b\leq a+1$ we have a nice description of the component of the moduli space of $\mu$-stable bundles where such bundles sit.
\begin{proposition}\label{serremoduli1}
Let $S=S(a,a,b)$ be a rational normal scroll with $b\leq a+1$. For each $\alpha\geq1$ there exists an irreducible component
\[
\MI_S^{\ 0}\bigl(\alpha H^2-(\alpha(a+b)+a-1)HF\bigr) \subseteq \MI_S\bigl(\alpha H^2-(\alpha(a+b)+a-1)HF\bigr)
\]
which is generically smooth of dimension $2c-10-2K_Sc_2(\sE)$ and containing all points corresponding to the bundles obtained in Proposition \ref{serreexistence2}
\end{proposition}
\begin{proof}
Let us begin by computing $\Ext^i(\sE,\sE)$. Since $\sE$ is $\mu$-stable, it is simple. Hence we have $\Ext^3(\sE,\sE)=0$ and $\Hom(\sE,\sE)\cong \mathbb{C}$. Now we show $\Ext^2(\sE,\sE)=0$. Take the short exact sequence \eqref{serreinsta} and tensor it by $\sE^\vee\cong \sE (H-(2a+b-2)F)$. Now taking cohomology we have
\[
H^2(S,\sE((1-a)F))\to \Ext_X^2(\sE,\sE)\to H^2(S,\sE \otimes \sI_{Y|X}(H-(a+b-1)F)).
\]
We show that both $H^2(S,\sE((1-a)F))$ and $H^2(S,\sE \otimes \sI_{Y|X}(H-(a+b-1)F))$ are zero. Take the cohomology of the short exact sequence \eqref{serreinsta} tensorized by $\sO_S((1-a)F)$. We obtain $H^2(S,\sE((1-a)F))\cong H^2(S, \sI_{Y|S})\cong H^1(Y,\sO_Y)\cong 0$ because $Y$ is the disjoint union of smooth rational curves. It remains to show that $H^2(S,\sE \otimes \sI_{Y|X}(H-(a+b-1)F))\cong 0 $.

In order to do so let us take the short exact sequence \eqref{serreinsta} and tensorize it by $\sO_S(H-(a+b-1)F)$. Taking cohomology we obtain $h^2(S,\sE(H-(a+b-1)F))=h^2(X,\sI_{Y|S}(H-bF))$. Now if we tensorize the short exact sequence
\begin{equation}\label{sesideal}
0\to\sI_{Y|S}\to\sO_S\to\sO_Y\to0
\end{equation}
by $\sO_S(H-bF)$ and we take cohomology we have $h^2(X,\sI_{Y|S}(H-bF))=h^1(Y,\sO_X(H-bF)\otimes \sO_Y)=0$ since $\sO_S(H-bF)$ restricts to each component of $Y$ to a degree $a-b$ line bundle. By hypothesis $a-b\geq 1$, thus we have $h^2(S,\sE(H-(a+b-1)F)=0$. Now if we take the cohomology of \eqref{sesideal} tensorized by $\sE(H-(a+b-1)F)$ we have
\[
h^2(S,\sE(H-(a+b-1)F)\otimes \sI_{Y|S})\leq h^1(Y,\sE(H-(a+b-1)F\otimes\sO_Y).
\]
But now using the fact that $\sE(H-(a+b-1)F\otimes \sO_Y\cong \sN_{Y/S}\cong\sO_\p(a-b)\oplus \sO_\p$ we have $h^1(Y,\sE(H-(a+b-1)F\otimes\sO_Y))=0$ and thus $h^2(X,\sE(H-(a+b-1)F)\otimes \sI_{Y|S})$. Finally we obtain $\Ext_X^2(\sE,\sE)=0$ and the assertion on the dimension of $\Ext_X^1(\sE,\sE)$ follows from Riemann-Roch, since $\sE$ is simple and $\chi (\sE \otimes \sE^\vee)=2K_Sc_2(\sE)-2c+10$.

The schemes as in \eqref{0locus2} represent points in a non-empty open subset $\mathcal{U}\subset\Gamma^{\times \alpha}(S)$.The latter product is a product of irreducible varieties, thus it is irreducible. It follows that $\sU$ is irreducible as well.

Since the vector bundle $\sE$ in Sequence \eqref{serreinsta} is uniquely determined by the scheme $Y$, we obtain in this way a flat family of bundles containing all the bundles obtained via Proposition \ref{serreexistence2} and parameterized by $\sU$. Thus all the vector bundles constructed via Construction \ref{disjointrational} lie in the same component $\MI_S^{0}(c_2(\sE)) \subseteq \MI_S(c_2(\sE))$.
\end{proof}
Combining Construction \ref{deformation} and Proposition \ref{serreexistence2} we are finally able to prove the following theorem
\begin{theorem}\label{esistenzatotale}
Let $S=S(a,a,b)$ be a rational normal scroll with $b\leq a+1$. Then for any $k_1\geq0$ and $k_2\geq 1-ck_1+a(k_1-1)$ there exists a $\mu$-stable instanton bundle with $c_2(\sE)=k_1H^2+k_2HF$ which is generically trivial on lines in the fibers and such that
\[
\Ext_X^1(\sE,\sE)=2\left(c-5+2k_1(c+1)+3k_2\right), \qquad \Ext_X^2(\sE,\sE)=\Ext_X^3(\sE,\sE)=0.
\]
In particular there exists, inside the moduli space $\MI(k_1H^2+k_2HF)$ of $\mu$-stable instanton bundles with $c_2(\sE)=k_1H^2+k_2HF$, a generically smooth, irreducible component of dimension $2\left(c-5+2k_1(c+1)+3k_2\right)$.
\end{theorem}
\begin{proof}
Observe that the bundles whose existence is guaranteed by Proposition \ref{serreexistence2} satisfy all the starting hypothesis of Construction \ref{deformation}. Thus the theorem follows by combining Proposition \ref{serreexistence2} and Construction \ref{deformation}.
\end{proof}

\section{Some cases with $a_0<a_1$}
In this section we deal with the existence of instanton bundles on $S=S(a_0,a_1,a_2)$ where $a_0<a_1$ and $a_0+a_1+a_2=c$. In this section we will prove the existence of instanton bundles on $S$, given some conditions on the invariants $a_i$'s. Let us start by describing the general hyperplane section of $S$. We will first recall the following definition.

\begin{definition}\cite[Definition 2.1]{Br}\label{relative}
The index of relative balance of a three-dimensional scroll $S(a_0,a_1,a_2)$ denoted $r(S)$ is given by
\[
r(S)=\min\left\{k\in\mathbb{N}: (3-k)a_{3-k}\leq\sum_{i=k}^3{a_{3-i}}\right\}.
\]
\end{definition}

Thus we are able to describe the general hyperplane section of $S$.

\begin{theorem}\cite[Theorem 2.4]{Br}\label{hsection}
Let $S=S(a_0,a_1,a_2)$ be a 3-dimensional rational normal scroll with index of relative balance $r(S)=r$. Then a general hyperplane section of $S$ is a 2-dimensional scroll with invariants $b_0\leq b_1$ satisfying
\[
b_0+b_1=a_0+a_1+a_2=c
\]
and the following conditions:
\begin{itemize}
\item $b_{2-r}\leq b_0+1$.
\item $b_{2-i}=a_{3-i}$ for all $1\leq i \leq r-1$.
\end{itemize}
\end{theorem}

Let us explicitely describe the  possible different cases. By Definition \ref{relative} it is clear that $r(S)\leq 2$ because $a_1 \leq a_1+a_0$. Indeed, in our case, $r(S)$ is either $1$ or $2$, i.e.
\begin{equation}
r(S)=\begin{sistema}
1 \quad \text{if $a_2 \leq a_0+a_1$,} \\
2 \quad \text{if $a_2>a_0+a_1$.}
\end{sistema}
\end{equation}

In particular by Theorem \ref{hsection} if $r(S)=1$ then we have that the general hyperplane section $H$ of $S$ is isomorphic to $S(b_0,b_1)$ with $b_0\leq b_1 \leq b_0+1$, i.e.
\begin{equation}
H=\begin{sistema}
S(b_0,b_0) \quad \ \ \ \ \ \text{if $c$ is even},\\
S(b_0,b_0+1) \quad \text{if $c$ is odd}.
\end{sistema}
\end{equation}
From now on we will focus on rational normal scrolls of even degree and with $r(S)=1$. In this case the general hyperplane section $H$ is isomorphic to $Q=\pp(\sO_\p \oplus \sO_\p)$ embedded as a scroll via $C_0+b_0f$, where as usual $C_0$ and $f$ represent the elements in the two rulings of $Q$ with the relations $C_0^2=f^2=0$ and $C_0f=1$. The curve $C_0$ is a rational normal curve of degree $b_0$ and its class in $A^2(S)$ is represented by $H^2-{b_0}HF$. Moreover we have the following exact sequence
\begin{equation}\label{sesc0}
0\to \sO_S(-2H+b_0F)\to \sO_S(-H)\oplus \sO_S(-H+b_0F) \to \sO_S \to \sO_{C_0} \to 0,
\end{equation}
from which we obtain $\sN_{C_0|S}\cong \sO_\p (b_0) \oplus \sO_\p$.

\begin{construction}\label{disjointc0}
Let us consider $D_1,D_2,\dots,D_\alpha$ disjoint rational normal curves from the family $H^2-b_0HF$. First of all observe that we can choose such curves so that they are all pairwise disjoint, since they are curves in one of the ruling of $\pp(\sO_\p(b_0)\oplus \sO_\p(b_0))\subset S$. Let us denote by Y the one-dimensional scheme
\begin{equation}\label{0locus3}
Y=\bigcup_{j=i}^{\alpha}{D_i}
\end{equation}
We claim that that $\det \mathcal{N}_{Y/S}\cong \sO_S(H)\otimes \sO_Y$. We can verify such an isomorphism component by component. We have that $\sN_{D_i|S}\cong\sO_\p(b_0) \oplus \sO_\p$. In particular
\[
\det \sN_{C_i/S} \otimes \sO_{C_i} \cong \sO_S(H) \otimes \sO_{C_i}.
\]
 i.e. the determinant of the normal bundle of $Y$ is extendable on $S$. Since $h^i(S,\sO_S(-H))=0$ for all $i$, it follows from Proposition \ref{HScorr} that there exists a unique (up to isomorphism) vector bundle  $\sF$ on $S$ with a section $s$ vanishing exactly along $Y$ and with $c_1(\sF)=H$ and $c_2(\sF)=Y$.
Thus $\sE=\sF(-H+(b_0-1)F)$ has $c_1(\sE)=-H+(c-2)F$ and $c_2(\sE)=\alpha H^2-(b_0(\alpha+1)-1)HF$ and it fits into
\begin{equation}\label{serreinsta2}
0\rightarrow \sO_S\left(-H+(b_0-1)F\right) \rightarrow \sE \rightarrow \sI_{Y|S}\left((b_0-1)F\right) \to 0.
\end{equation}
\end{construction}

\begin{proposition}\label{serreexistence3}
Let $\sE$ be a vector bundle on $S(a_0,a_1,a_2)$ with $a_0+a_1+a_2=2b_0$ and $a_2<a_0+a_1$ with Chern classes $c_1(\sE)=-H+(c-2)F$, $c_2(\sE)=\alpha H^2-(b_0(\alpha+1)-1)HF$ with $\alpha\geq 1$. If $\sE(H-(b_0-1)F)$ has a section vanishing along $Y$ as in Construction \ref{disjointc0}, i.e. $\sE$ fits into
\begin{equation}\label{serreinsta2}
0\rightarrow \sO_S\left(-H+(b_0-1)F\right) \rightarrow \sE \rightarrow \sI_{Y|S}\left((b_0-1)F\right) \to 0,
\end{equation}
then $\sE$ is a $\mu$-stable instanton bundle with charge $k=b_0(\alpha-1)+1$, such that $\sE_L\cong \sO_\p \oplus \sO_\p(-1)$ for the generic line L in the fibers.
\end{proposition}
\begin{proof}
By Construction \ref{disjointc0} we have $c_1(\sE)=-H+(c-2)F$ and $c_2(\sE)=\alpha H^2-(b_0(\alpha+1)-1)HF$.

First we check that $\sE$ has no section. Indeed taking the cohomology of \eqref{serreinsta2} we obtain $H^0(S,\sE)=H^0(Y,\sI_{Y|S}\bigl((b_0-1)F\bigr))$. However $Y$ is not contained in any surface in the class $(b_0-1)F$. Indeed any $D_j$ intersect each fiber in one point, so $Y \not\subset (b_0-1)F$ and $H^0(Y,\sI_Y((b_0-1)F))=0$.

Now we prove that $\sE$ satisfies the instantonic condition. If we tensor \eqref{serreinsta2} by $\sO_X(-H)$ we obtain $H^i(S,\sE(-H))=H^i(Y,\sI_{Y|S}\bigl(-H+(b_0-1)F\bigr))=0$ for all $i$.

The next step is to prove the $\mu$-stability, and in order to do so we will use Proposition \ref{hoppe}. Let $D=\gamma H+\beta F$ be a divisor on $S$ such that $2\gamma b_0+\beta\leq 1$. We want to show that $H^0(S,\sE(D))=0$.
Let us begin by tensoring \eqref{serreinsta2} by $\sO_S(D)$. Now we show that $H^0(S,\sO_S\bigl((\gamma-1)H+(b_0+\beta-1)F)\bigr)=0$ if $2\gamma b_0+\beta \leq 1$.

It is clear that when $\gamma<1$ this group vanishes by Lemma \ref{lem0}. Now suppose $\gamma \geq 1$. By Lemma \ref{lem0} we have $H^0(S,(\sO_S\bigl((\gamma-1)H+(b_0+\beta-1)F\bigr)))=0$ if and only if $a_2(\gamma-1)+b_0+\beta-1<0$. But $\beta \leq 1-2\gamma b_0 $, thus $a_2(\gamma-1)+b_0+\beta-1\leq(\gamma-1)(a_2-2b_0)-b_0<0$ and $H^0(S,\sO_S\bigl((\gamma-1)H+(b_0+\beta-1)F\bigr))=0$.

Now we show that $H^0(Y,\sI_Y(\gamma H+(b_0-1+\beta)F))=0$. In order to do so we tensor $$0\to \sI_Y \to \sO_S \to\sO_Y\to0$$ by $\sO_S( \gamma H+(b_0-1+\beta)F)$ and consider the long exact sequence induced in cohomology. In particular $H^0(Y,\sI_Y(\gamma H+(b_0-1+\beta)F))\leq H^0(S,\sO_S(\gamma H+(b_0-1+\beta)F))$ and by Lemma \ref{lem0} the latter vanishes for $\gamma<0$. Now suppose $\gamma=0$. In this case $H^0(Y,S,\sI_Y((b_0-1+\beta)F))=0$, since $Y$ is not contained in any surface in the class $(b_0-1+\beta)F$. It remains to consider the case $\gamma>0$. In this cases, by Lemma \ref{lem0}, $H^0(\sO_S(\gamma H+(b_0-1+\beta)F))=0$ if and only if $\gamma a_2+b_0-1+\beta <0$. We have $\beta \leq1-2\gamma b_0$ thus
 $\gamma a_2+b_0-1+\beta\leq \gamma a_2+b_0-1+1-2\gamma b_0=\gamma(a_2-2b_0)+b_0=\gamma(a_2-c)+\frac{c}{2}$. Recall by hypothesis that $a_2< \frac{c}{2}$ thus $\gamma(a_2-c)+\frac{c}{2}<0$ and $H^0(Y,\sI_Y(\gamma H+(b_0-1+\beta)F))=0$. Finally we have $H^0(S,\sE(\gamma H+\beta F))=0$ if $\gamma c+\beta \leq1$, thus $\sE$ is stable.

To conclude this proof we show that $\sE_L\cong \sO_\p \oplus \sO_\p(-1)$ for the generic line L in the fibers. Indeed let us consider a fiber $F$ and tensor \eqref{serreinsta2} by $\sO_F$. Taking the long exact sequence in cohomology we get
\[
h^0(F,\sE \otimes \sO_{F})\leq h^0(F,\sO_{F}(-1))+h^0(F,\sI_{Y\cap F|F}).
\]
But $h^0(F,\sO_{F}(-1))=0$ since $F$ is isomorphic to $\pp^2$. Moreover each component of $Y$ intersect transversely each fiber $F$, thus $Y\cap F\neq \emptyset$ and $h^0(F,\sI_{Y\cap F|F})=0$. In particular we have $h^0(F,\sE \otimes \sO_{F})=0$. $\sE \otimes \sO_{F}$ is a normalized rank two vector bundle on $\pp^2$ , i.e. $c_1(\sE \otimes \sO_{F})=-1$, which has no section, thus it is stable. By Grauert-M\"ulich Theorem $\sE \otimes \sO_{F}$ is isomorphic to $\sO_\p\oplus \sO_\p(-1)$ when restricted to the generic line in the fibers.
\end{proof}
In the following proposition we study the component of the moduli space of $\mu$-stable bundle containing the vector bundles which existence is guaranteed by Proposition \ref{serreexistence3}.
\begin{proposition}
Let $\sE$ be a vector bundle on $S(a_0,a_1,a_2)$ with $a_0+a_1+a_2=2b_0$ and $a_2<a_0+a_1$ with Chern classes $c_1(\sE)=-H+(c-2)F$, $c_2(\sE)=\alpha H^2-(b_0(\alpha+1)-1)HF$ with $\alpha\geq 1$. For each $\alpha\geq1$ there exists an irreducible component
\[
\MI_S^{\ 0}\bigl(\alpha H^2-(b_0(\alpha+1)-1)HF\bigr) \subseteq \MI_S\bigl(\alpha H^2-(b_0(\alpha+1)-1)HF\bigr)
\]
which is generically smooth of dimension $2c-10-2K_Sc_2(\sE)$ and containing all points corresponding to the bundles obtained in Proposition \ref{serreexistence3}
\end{proposition}
\begin{proof}
We omit the proof of this proposition because the strategy is analougous to the one of Proposition \ref{serremoduli1}.
\end{proof}
Combining Construction \ref{deformation} and Proposition \ref{serreexistence3} we are finally able to prove the following theorem
\begin{theorem}\label{esistenzatotale2}
Let $S=S(a_0,a_1,a_2)$ be a rational normal scroll of even degree with $a_0+a_1>a_2$. Then for any $k_1\geq0$ and $k_2\geq 1-ck_1+b_0(k_1+1)$ there exists a $\mu$-stable instanton bundle with $c_2(\sE)=k_1H^2+k_2HF$ which is generically trivial on lines in the fibers and such that
\[
\Ext_X^1(\sE,\sE)=2\left(c-5+2k_1(c+1)+3k_2\right), \qquad \Ext_X^2(\sE,\sE)=\Ext_X^3(\sE,\sE)=0.
\]

In particular there exists, inside the moduli space $\MI(k_1H^2+k_2HF)$ of $\mu$-stable instanton bundles with $c_2(\sE)=k_1H^2+k_2HF$, a generically smooth, irreducible component of dimension $2\left(c-5+2k_1(c+1)+3k_2\right)$.
\end{theorem}
\begin{proof}
Observe that the bundles whose existence is guaranteed by Proposition \ref{serreexistence3} satisfy all the starting hypothesis of Construction \ref{deformation}. Thus the theorem follows by combining Proposition \ref{serreexistence3} and Construction \ref{deformation}.
\end{proof}

\bibliographystyle{amsplain}

\begin{thebibliography}{10}

\bibitem{AO3}
V.~Ancona and G.~Ottaviani, \emph{Canonical resolutions of sheaves on
  {Schubert} and {Brieskorn} varieties}, Complex Analysis (Klas Diederich,
  ed.), 1991, pp.~14--19.

\bibitem{AM}
V.~Antonelli and F.~Malaspina, \emph{Instanton bundles on the {Segre} threefold
  with {Picard} number three}, Mathematische Nachrichten \textbf{293} (2020),
  no.~6, 1026--1043.

\bibitem{AHMP}
M.~Aprodu, S.~Huh, F.~Malaspina, and J.~Pons-Llopis, \emph{Ulrich bundles on
  smooth projective varieties of minimal degree}, Proc. Amer. Math. Soc.
  \textbf{147} (2019), 5117--5129.

\bibitem{arrondo}
E.~Arrondo, \emph{A {Home-Made} {Hartshorne-Serre} correspondence}, Rev. Mat.
  Complut. \textbf{20} (2007), no.~2, 423 -- 443.

\bibitem{Ba2}
W.~Barth, \emph{Moduli of vector bundles on the projective plane}, Inventiones
  mathematicae \textbf{42} (1977), no.~1, 63--91.

\bibitem{Bea}
A.~Beauville, \emph{Determinal hypersurface}, Michigan Math. J. \textbf{48}
  (2000), 39--64.

\bibitem{Be}
A.~A. Beilinson, \emph{Coherent sheaves on {Pn} and problems of linear
  algebra}, Functional Analysis and Its Applications \textbf{12} (1978), no.~3,
  214--216.

\bibitem{BO}
C.~B\"ohning, \emph{Derived categories of coherent sheaves on rational
  homogeneous manifold}, Doc. Math. \textbf{11} (2006), 261--331.

\bibitem{Br}
J.~N. Brawner, \emph{Tetragonal curves, scrolls and {$K3$} surfaces}, Trans.
  Amer. Math. Soc. \textbf{349} (1997), no.~8, 3075--3091.

\bibitem{CCGM}
G.~Casnati, E.~Coskun, {\"O. Gen\c{c}}, and F.~Malaspina, \emph{Instanton
  bundles on the blow-up of the projective $3$-space at a point}, To appear in
  Mich. Math. Journal (2019).

\bibitem{CG}
G.~Casnati and {\"O. Gen\c{c}}, \emph{Instanton bundles on two Fano threefolds
  of index $1$.}, [arXiv:2004.03415], to appear in Forum Matematicum (2020).

\bibitem{CO}
L.~Costa and G.~Ottaviani, \emph{Nondegenerate multidimensional matrices and
  instanton bundles}, Trans. Amer. Math. Soc. \textbf{355} (2003), no.~1,
  49--55.

\bibitem{CoMi}
L. Costa and R. M. Mir{\'o}-Roig, \emph{Monads and instanton bundles on smooth hyperquadrics}, Mathematische Nachrichten \textbf{282} (2009),  no.~2, 169--179.

\bibitem{CoMiRo2020}
L. Costa and R. M. Mir{\'o}-Roig, \emph{Instanton bundles vs Ulrich
  bundles on projective spaces}, Beitr{\"a}ge zur Algebra und Geometrie /
  Contributions to Algebra and Geometry (2020).

\bibitem{Don}
S.~K. Donaldson, \emph{Vector bundles on flag manifolds and the {W}ard
  correspondence}, Geometry today (Rome), Progr. Math., vol.~60,
  {Birkh\"{a}user Boston, Boston, MA}, 1985, pp.~109--119.

\bibitem{EH}
D.~Eisenbud and J.~Harris, \emph{On varieties of minimal degree (a centennial
  account)}, Algebraic Geometry (Bowdoin, ed.), 1985.

\bibitem{EH2}
\bysame, \emph{{$3264$} and {All That}. {Intersection Theory in Algebraic
  Geometry}}, Cambridge University Press, 2016.

\bibitem{ESW}
D.~Eisenbud, F.~O. Schreyer, and J.~Weyman, \emph{Resultants and {Chow} forms
  via exterior syzygies}, J. Amer. Math. Soc. \textbf{16} (2003), no.~2,
  537--579.

\bibitem{ES}
G.~Ellingsrud and S.~A. Str{\o}mme, \emph{Stable rank-2 vector bundles on
  {$\mathbb{P}^3$} with {$c_1=0$} and {$c_2=3$}}, Mathematische Annalen
  \textbf{255} (1981), no.~1, 123--135.

\bibitem{Fa}
D.~Faenzi, \emph{Even and odd instanton bundles on {Fano} threefolds of
  {Picard} number one}, Manuscripta Mathematica \textbf{144} (2014), no.~1,
  199--239.

\bibitem{fulton}
W.~Fulton, \emph{Intersection theory}, Springer, 1984.

\bibitem{GO}
A.~Gorodentsev and S.A. Kuleshov, \emph{Helix theory}, Mosc. Math. J.
  \textbf{4} (2004), no.~2, 377--440.

\bibitem{Har2}
R.~Hartshorne, \emph{Stable vector bundles of rank 2 on {$\mathbb{P}^3$}},
  Mathematische Annalen \textbf{238} (1978), 229--280.

\bibitem{HH}
R.~Hartshorne and A.~Hirschowitz, \emph{Cohomology of a general instanton
  bundle}, Annales scientifiques de l'E.N.S. \textbf{15} (1982), no.~2,
  199--207.

\bibitem{Hit}
N.~J. Hitchin, \emph{K{\"{a}}hlerian twistor spaces}, Proc. Lond. Math. Soc.
  \textbf{43} (1981), no.~1, 133--150.

\bibitem{JMPS}
M.~Jardim, G.~Menet, D.~Prata, and H.~S{\'a} Earp, \emph{Holomorphic bundles
  for higher dimensional gauge theory}, Bull. London Math. Soc. \textbf{49}
  (2017), 117--132.

\bibitem{JV}
M.~Jardim and M.~Verbitsky, \emph{Trihyperk\"{a}hler reduction and instanton
  bundles on {$\mathbb{CP}^3$}}, Compositio Mathematica \textbf{150} (2014),
  no.~11, 1836--1868.

\bibitem{Kat}
P.~I. Katsylo, \emph{Rationality of the module variety of mathematical
  instantons with {$c_2=5$}}, Lie groups, their discrete subgroups, and
  invariant theory, Adv. Math. Sci., vol.~8, Amer. Math. Soc., 1992,
  pp.~105--111.

\bibitem{KRR}
N.~Kumar, A.~Rao, and G.~Ravindra, \emph{Arithmetically cohen-macaulay bundles
  on threefold hypersurfaces}, Comment. Math. Helv. \textbf{82} (2006).

\bibitem{Kuz}
A.~Kuznetsov, \emph{Instanton bundles on {Fano} threefolds}, Central European
  Journal of Mathematics \textbf{4} (2012), 1198--1231.

\bibitem{MMP}
F.~Malaspina, S.~Marchesi, and Joan Pons-Llopis, \emph{Instanton bundles on the
  flag variety {$F(0,1,2)$}}, arXiv:1706.06353 [math.AG], to appear on {Ann
  Scuola Norm-Sci}, 2018.

\bibitem{Mar}
M. Maruyama: {\em Boundedness of semistable sheaves of small ranks}. Nagoya Math. J. \textbf{78} (1980), 65--94.

\bibitem{MR}
R. M. Mir{\'o}-Roig, \emph{The representation type of rational normal scrolls}, Rend. Circ. Mat. Palermo (2013), \textbf{62} (2013), 153--164

\bibitem{OSS}
C.~Okonek, M.~Schneider, and H.~Spindler, \emph{Vector bundles on complex
  projective spaces}, Progress in Mathematics, vol.~3, Springer-Verlag New
  York, 1980.

\bibitem{RU}
A.~N. Rudakov, \emph{Helices and vector bundles}, Seminaire Rudakov, vol. 148,
  Cambridge University Press, 1990.

\bibitem{Sa}
G.~Sanna, \emph{Small charge instantons and jumping line on quintic del
  {P}ezzo}, Int. Math. Res. Notices \textbf{21} (2017), 6523--6583.

\bibitem{T1}
A.~S. Tikhomirov, \emph{Moduli of mathematical instanton vector bundles with
  odd $c2$ on projective space}, Izv. Math. \textbf{76} (2012), no.~5.

\bibitem{T2}
\bysame, \emph{Moduli of mathematical instanton vector bundles with even $c2$
  on projective space}, Izv. Math. \textbf{77} (2013), no.~6.

\end{thebibliography}

\bigskip
\noindent
Vincenzo Antonelli,\\
Dipartimento di Scienze Matematiche, Politecnico di Torino,\\
c.so Duca degli Abruzzi 24,\\
10129 Torino, Italy\\
e-mail: {\tt vincenzo.antonelli@polito.it}

\bigskip
\noindent
Francesco Malaspina,\\
Dipartimento di Scienze Matematiche, Politecnico di Torino,\\
c.so Duca degli Abruzzi 24,\\
10129 Torino, Italy\\
e-mail: {\tt francesco.malaspina@polito.it}

\end{document}